\documentclass{article}%
\usepackage{amsmath}
\usepackage{amsfonts}
\usepackage{amssymb}
\usepackage{graphicx}
\usepackage[body={15.5cm,21cm}, top=3cm]{geometry}
\usepackage[titletoc]{appendix}
\usepackage{titlesec}
\usepackage{titletoc}
\usepackage{setspace}
\usepackage{hyperref}
\usepackage{booktabs}%
\providecommand{\U}[1]{\protect\rule{.1in}{.1in}}
\renewcommand{\emph}[1]{\textcolor{red}{#1}}

\hypersetup{
colorlinks=true,
linkcolor=blue,
anchorcolor=blue,
citecolor=blue}
\numberwithin{equation}{section}
\newtheorem{theorem}{Theorem}[section]

\newtheorem{condition}[theorem]{Assumption}

\newtheorem{definition}[theorem]{Definition}
\newtheorem{example}[theorem]{Example}

\newtheorem{lemma}[theorem]{Lemma}
\newtheorem{notation}[theorem]{Notation}

\newtheorem{proposition}[theorem]{Proposition}
\newtheorem{remark}[theorem]{Remark}

\newenvironment{proof}[1][Proof]{\noindent\textbf{#1.} }{\ \rule{0.5em}{0.5em}}
\begin{document}

\title{Infinite horizon quadratic backward stochastic differential equations driven
by $G$-Brownian motion}
\author{Yiqing Lin \thanks{MOE-LSC and School of Mathematical Sciences, Shanghai Jiao Tong University, Shanghai, 200240, China. \texttt{yiqing.lin@sjtu.edu.cn}}
\and Yifan Sun \thanks{Department of Applied Mathematics, The Hong Kong Polytechnic
University, Hong Kong, China. \texttt{yifansun1111@gmail.com}.}
\and Falei Wang\thanks{Zhongtai Securities Institute for Financial Studies, Shandong University, Jinan 250100, China.
\texttt{flwang2011@gmail.com}.}}
\date{}
\maketitle

\begin{abstract}
The aim is to prove the well-posedness of infinite horizon backward stochastic
differential equations driven by $G$-Brownian motion ($G$-BSDEs) with
quadratic generators. To this end, we provide a full construction of explicit
solutions to linear $G$-BSDEs with unbounded coefficients and the
linearization method under the quadratic assumption. In addition, the comparison
theorems for both finite and infinite horizon $G$-BSDEs are established.

\medskip



\end{abstract}

\section{Introduction}

This paper investigates the following type of backward stochastic differential
equations driven by $G$-Brownian motion ($G$-BSDEs) on the infinite horizon:
for any $0\leq t\leq T<\infty$,%
\begin{equation}
Y_{t}=Y_{T}+\int_{t}^{T}f\left(  s,Y_{s},Z_{s}\right)  ds+\sum_{i,j=1}^{d}%
\int_{t}^{T}g^{ij}\left(  s,Y_{s},Z_{s}\right)  d\left\langle B\right\rangle
_{s}^{ij}-\int_{t}^{T}Z_{s}dB_{s}-\left(  K_{T}-K_{t}\right)  , \label{i0}%
\end{equation}
where the generators are supposed to be quadratic in the $z$-component. The
purpose is to find a triplet $(Y,Z,K)$ that satisfies this equation for any
$0\leq t\leq T<\infty$.

In the classical probability framework, nonlinear backward stochastic
differential equations (BSDEs) on finite horizon, firstly introduced by
Pardoux and Peng \cite{PP1990} in 1990, have become essential tools in several
fields such as stochastic control, partial differential equations (PDEs) and
mathematical finance. There has been significant interest in weakening the
conditions for classical BSDEs on finite horizon under the Lipschitz
assumption and applying generalized BSDEs to various problems.

In 1997, Kobylanski \cite{K2000} first explored BSDEs whose generators have
quadratic growth in the $z$-component (quadratic BSDEs) using exponential
change and approximation methods. Subsequently, Lepeltier and San Martin
\cite{LS1998} established the existence result for BSDEs with
superlinear-quadratic generators. For cases with unbounded terminal
conditions, Briand and Hu \cite{BH2006} addressed the solvability of this type
of quadratic BSDEs and, later in \cite{BH2008}, they proved the uniqueness of
the solutions under the assumption that the generators are convex (or
concave). In addition to the research of well-posedness under various
conditions, quadratic BSDEs have been also extensively studied for their
applications, such as utility maximization, indifference pricing and financial
market equilibrium problems (see, e.g.,
\cite{AIR2010,BE2004,HIM2005,HLT2024,XZ2018}). For further results related to
quadratic BSDEs, we refer to \cite{AIR2010,BE2013,FHT2023,L2020b,T2008} among others.

The study of BSDEs has also been extended to the context of infinite horizon.
Fuhrman and Tessitore \cite{FT2004} and Hu and Tessitore \cite{HT2007}
investigated infinite horizon BSDEs in Hilbert spaces based on the connection
between BSDEs and elliptic PDEs. Briand and Confortola \cite{BC2008a} later
established the well-posedness of quadratic BSDEs on the infinite time
horizon. A special type of infinite horizon BSDEs, called ergodic BSDEs, was
formulated by Fuhrman, Hu and Tessitore \cite{FHT2009} motivated by the
ergodic stochastic control problem, and it can be applied to study the
long-time behavior of solutions to parabolic PDEs (see, e.g., \cite{CFP2016}).
More recently, it was shown by Liang and Zariphopoulou \cite{LZ2017} that
ergodic BSDEs with quadratic generators serve as a useful tool for
constructing homothetic examples of forward performance processes, which are
the forward counterparts of value function processes in utility maximization
problems. This connection has also been studied by \cite{HLT2020,JL2023}.

The aforementioned literature on BSDEs is confined to the context of classical
probability space. Recently, BSDE theory has been explored in the nonlinear
expectation framework, particularly in the $G$-expectation space proposed by
Peng \cite{Peng2007,Peng2008b}. The $G$-expectation, a typical type of
sublinear expectation, is constructed via the so-called $G$-heat equations on
the canonical space, with the associated canonical process being a
$G$-Brownian motion. The $G$-It\^{o}'s calculus is also developed. We refer
the readers to Peng's book \cite{PengBOOK} and related references for more
introduction and development of $G$-expectation theory.

Under finite horizon setting, Hu, Ji, Peng and Song \cite{HJPS2014a} showed
the well-posedness of $G$-BSDEs under the Lipschitz condition using a PDE
method and Galerkin approximation. Later, Hu, Lin and Soumana Hima
\cite{HLS2018} established the existence of the solution to quadratic (in
$z$-component) $G$-BSDEs via a discretization approach and proved the
uniqueness via a priori estimates based on the concept of $G$-BMO martingale
generators and related $G$-Girsanov's theorem. More recently,
\cite{CT2020,GLX2023,HTW2022} have also contributed to quadratic $G$-BSDEs
under various settings and conditions. Regarding infinite horizon cases, Hu
and Wang \cite{HW2018} and Hu, Sun and Wang \cite{HSW2022} studied the
infinite horizon and ergodic BSDEs driven by $G$-Brownian motion and within a
general nonlinear expectation space, respectively. Subsequently, Sun and Wang
\cite{SW2024} dealt with the Markovian infinite horizon $G$-BSDEs and ergodic
$G$-BSDEs with quadratic generators by a truncation method and accordingly
generalized robust forward performance processes to the $G$-framework.

It is worth mentioning that second order BSDEs, proposed by Soner, Touzi and
Zhang in \cite{STZ2012,STZ2013}, share a lot of similarities with $G$-BSDEs.
This type of BSDEs has been further developed by Possama\"{\i}, Zhou
\cite{PZ2013} and Lin \cite{L2016} under the quadratic condition. We refer the
reader to \cite{LRTY2020,MPZ2015,PTZ2018} and the references therein for more
discussions on second order BSDEs.

In this paper, the objective is to investigate the well-posedness of the BSDE
(\ref{i0}). To this end, similarly to previous works
\cite{BH1998,HSW2022,HW2018}, we first construct the explicit solutions to
linear BSDEs, then obtain the estimates of the difference of two BSDEs using
linearization, and finally investigate the convergence of a sequence of finite
horizon BSDEs as the horizon approaches infinity, with the limit being the
desired solution to the infinite horizon BSDE. However, all steps require
reconsideration due to the $G$-expectation setting and the quadratic assumptions.

We first study the linear $G$-BSDEs with unbounded coefficients, aiming to
provide an explicit representation of the solutions. Specifically, we deal
with the case when the coefficients for $z$-components are $G$-BMO martingale
generators. We recall that Theorem 3.2 in \cite{HJPS2014b} gives the explicit
solution to linear BSDEs driven by $1$-dimensional $G$-Brownian motion with
uniformly bounded coefficients via dual stochastic differential equation
driven by $G$-Brownian motion ($G$-SDE) in an auxiliary extended
$2$-dimensional $\tilde{G}$-expectation space. However, this result is not
applicable in our setting due to the unboundedness of the coefficients, and
the high dimensionality of the $G$-Brownian motion also increases the
complexity of the problem. To address this difficulty, we employ the dual
method in an extended space and then, with the help of $G$-Girsanov's formula
with respect to $G$-BMO martingale generators, seek to eliminate the unknown
parts via taking conditional expectations. The space extension for the
$1$-dimensional case is similar to that in \cite{HJPS2014b}. However,
constructing the extended space for the general $G$-framework with higher
dimensionality is more complex. For clarity, we detail the construction
procedures for the cases $d=1$, $d=2$, and the general case, respectively. We
then represent the solutions to linear $G$-BSDEs as conditional expectations
transformed by Girsanov's formula in the extended space.

A linearization method for quadratic generators is derived with the help of
properties of $G$-BMO martingale generators, and we also obtain an estimate
for a type of $G$-BSDEs. Combining aforementioned technical lemmas, we
establish the comparison theorems for both finite and infinite horizon
quadratic $G$-BSDEs and, finally, the existence and uniqueness result for
infinite horizon cases.

Our contribution in this paper is threefold. Firstly, we provide a detailed
construction of solutions to linear BSDEs driven by multi-dimensional
$G$-Brownian motion, extending existing results from cases with bounded
coefficients to those with unbounded coefficients. Secondly, we refine the
linearization method under quadratic conditions, which complements the theory.
Thirdly, by combining the above techniques, we establish the solvability for
infinite horizon quadratic $G$-BSDEs as well as comparison theorems for both
finite and infinite horizons.

This paper is organized as follows. In Section \ref{Pre}, we introduce the
basic setups and results of the $G$-expectation framework especially those
related to quadratic $G$-BSDEs. Section \ref{Explicit solution} is devoted to
the explicit solutions of linear $G$-BSDEs with unbounded coefficients on
finite horizon, and the well-posedness of infinite horizon quadratic $G$-BSDEs
is established in Section \ref{infQGBSDE}.

\section{Preliminaries\label{Pre}}

Throughout this paper, we will make use of the following notations:

\begin{itemize}
\item On the Euclidean space $\mathbb{R}^{n}$, $xy:=$ scalar product of
$x,y\in\mathbb{R}^{n}$ and $|x|:=$ norm of $x\in\mathbb{R}^{n}$;

\item $I_{d}:=d\times d$ identity matrix;

\item $J_{d}:=d\times d$ constant matrix where every entry is $1$;

\item $O_{m\times n}:=m\times n$ zero matrix and $O_{d}:=O_{d\times d}$;

\item $\mathbb{S}(d):=$ the space of all $d\times d$ symmetric matrices;

\item For $A\in\mathbb{R}^{d\times d}$, $\operatorname*{tr}[A]:=$ trace of $A$;

\item For $A\in\mathbb{R}^{n\times m}$, $A^{\top}:=$ transpose of $A$;

\item For $A\in\mathbb{R}^{d\times d}$ of full rank $d$, $A^{-1}:=$ inverse of
$A$, and $A^{-1,\top}:=(A^{-1})^{\top}=(A^{\top})^{-1}.$
\end{itemize}

Now we recall some definitions and results under $G$-framework which will be
frequently used in the sequel.

\subsection{$G$-Brownian motion and $G$-stochastic integrals}

We will work throughout with the space $\Omega:=\{\omega\in C_{0}%
([0,\infty);\mathbb{R}^{d}),\omega_{0}=0\}$ equipped with the distance%
\[
\rho\left(  \omega^{1},\omega^{2}\right)  :=\sum_{N=1}^{\infty}2^{-N}\left[
\left(  \max_{t\in\left[  0,N\right]  }\left\vert \omega_{t}^{1}-\omega
_{t}^{2}\right\vert \right)  \wedge1\right]  ,
\]
where $C_{0}([0,\infty);\mathbb{R}^{d})$ is the space of all $\mathbb{R}^{d}%
$-valued continuous functions. On this space, let $B$ be the canonical
process, $\mathbb{P}_{0}$ the Wiener measure, $\mathbb{F}:=\{\mathcal{F}%
_{t}\}_{t\geq0}$ the natural filtration generated by $B$ and $\mathcal{B}%
(\Omega)$ the Borel $\sigma$-algebra of $\Omega$. For each fixed $T>0$, we
write $\Omega_{T}:=\{\omega_{\cdot\wedge T},\omega\in\Omega\}$ and set
\begin{align*}
L_{ip}(\Omega_{T})  &  :=\left\{  \varphi(B_{t_{1}\wedge T},\ldots
,B_{t_{n}\wedge T}),n\geq1,t_{1},\ldots,t_{n}\in\lbrack0,\infty),\varphi\in
C_{b,Lip}(\mathbb{R}^{d\times n})\right\}  ,\\
L_{ip}(\Omega)  &  :=\bigcup\nolimits_{T>0}L_{ip}(\Omega_{T}),
\end{align*}
with $C_{b,Lip}(\mathbb{R}^{d\times n})$ being the space of all $\mathbb{R}%
$-valued bounded Lipschitz functions on $\mathbb{R}^{d\times n}$.

For a given monotonic and sublinear function $G:\mathbb{S}(d)\rightarrow
\mathbb{R}$, Peng \cite{Peng2007,Peng2008b} constructed a consistent sublinear
expectation $(\mathbb{\hat{E}}_{t})_{t\geq0}$ on $(\Omega,L_{ip}(\Omega))$ via
a fully nonlinear PDE
\[
\left\{
\begin{array}
[c]{l}%
\partial_{t}u\left(  t,x\right)  -G\left(  \partial_{xx}^{2}u\left(
t,x\right)  \right)  =0,\text{ }\left(  t,x\right)  \in\lbrack0,\infty
)\times\mathbb{R}^{d},\\
u\left(  0,x\right)  =\varphi\left(  x\right)  ,
\end{array}
\right.
\]
which is called $G$-heat equation. More precisely, $\mathbb{\hat{E}%
}:=\mathbb{\hat{E}}_{0}$ satisfies $\mathbb{\hat{E}}[\varphi(x+B_{t})]=u(t,x)$
for all $(t,x)\in\lbrack0,\infty)\times\mathbb{R}^{d}$. We call $\mathbb{\hat
{E}}:=\mathbb{\hat{E}}_{0}$ the $G$-expectation and $\mathbb{\hat{E}}_t$ the
conditional $G$-expectation at $t$. For $p\geq1$, denote by $\mathbb{L}%
_{G}^{p}(\Omega)$ (resp., $\mathbb{L}_{G}^{p}(\Omega_{T})$) the completion of
$L_{ip}(\Omega)$ (resp., $L_{ip}(\Omega_{T})$) under the norm $\Vert\cdot
\Vert_{\mathbb{L}_{G}^{p}}:=(\mathbb{\hat{E}}[|\cdot|^{p}])^{1/p}$, and, in
turn, $(\mathbb{\hat{E}}_{t})_{t\geq0}$ can be continuously extended to
$\mathbb{L}_{G}^{1}(\Omega)$ under $\Vert\cdot\Vert_{\mathbb{L}_{G}^{1}}$. The
complete space $(\Omega,\mathbb{L}_{G}^{1}(\Omega),\mathbb{\hat{E}})$ is
called the $G$-expectation space, on which the canonical process
$B=(B^{i})_{i=1}^{d}$ is called a $d$-dimensional $G$-Brownian motion. Herein,
we only consider the non-degenerate case, namely, there exists a
${\underline{\sigma}}>0$ such that $G(A)-G(B)\geq\frac{1}{2}{\underline
{\sigma}}^{2}\operatorname*{tr}[A-B]$ for any $A\geq B$. In this case, there
exists a bounded, convex and closed set $\Theta\subset\mathbb{R}^{d\times d}$
of invertible matrices such that%
\begin{equation}
G\left(  A\right)  =\frac{1}{2}\sup_{\theta\in\Theta}\operatorname*{tr}\left[
\theta\theta^{\top}A\right] ~ \text{ for }A\in\mathbb{S}(d). \label{G}%
\end{equation}

For $G$-Brownian motion $B=(B^{i})_{i=1}^{d}$, its quadratic variation process
matrix $\langle B\rangle:=(\langle B\rangle^{ij})_{i,j=1}^{d}$ is defined by,
for $t\geq0$,%
\[
\left\langle B\right\rangle _{t}^{ij}:=\left\langle B^{i},B^{j}\right\rangle
_{t}:=\lim_{\mu(\pi_{t}^{N})\rightarrow0}\sum_{k=0}^{N-1}\left(
B_{t_{k+1}^{N}}^{i}-B_{t_{k}^{N}}^{i}\right)  \left(  B_{t_{k+1}^{N}}%
^{j}-B_{t_{k}^{N}}^{j}\right)  =B_{t}^{i}B_{t}^{j}-\int_{0}^{t}B_{s}^{i}%
dB_{s}^{j}-\int_{0}^{t}B_{s}^{j}dB_{s}^{i},
\]
where $\pi_{t}^{N}:=\{0=t_{0}^{N},t_{1}^{N},\ldots,t_{N}^{N}=t\}$,
$N=1,2,\ldots,$ is a sequence of partitions of $[0,t]$ and $\mu(\pi_{t}%
^{N}):=\max\{|t_{i+1}^{N}-t_{i}^{N}|,i=0,1,\ldots,N-1\}$. $\langle
B\rangle^{ij}$ is called the mutual variation process of $B^{i}$ and $B^{j}$.
According to \cite[Corollary 3.5.8]{PengBOOK}, we know that $\langle
B\rangle_{t}\in\{t\theta\theta^{T},\theta\in\Theta\}$ and, for any
$i,j=1,\ldots,d$, there exist $\bar{\sigma}_{ij}^{2}\geq0$ such that
$d|\langle B\rangle_{t}^{ij}|\leq\bar{\sigma}_{ij}^{2}dt$. For convenience we
denote
\[
\bar{\sigma}_{\Sigma}^{2}:=\sum_{i,j=1}^{d}\bar{\sigma}_{ij}^{2}. %
\]

It was shown by \cite[Theorem 52]{DHP2011} that the $G$-expectation can be
represented as an upper expectation%
\[
\mathbb{\hat{E}}\left[  X\right]  =\sup_{\mathbb{P}\in\mathcal{P}_{G}%
}\mathbb{E}^{\mathbb{P}}\left[  X\right] ~ \text{ for all }X\in\mathbb{L}%
_{G}^{1}(\Omega),
\]
where $\mathcal{P}_{G}$ is a weakly compact family of probability measures on
$(\Omega,\mathcal{B}(\Omega))$. We can accordingly define the Choquet capacity
associated to $\mathcal{P}_{G}$ by%
\[
c\left(  A\right)  :=\sup_{\mathbb{P}\in\mathcal{P}_{G}}\mathbb{P}\left(
A\right) ~ \text{ for }A\in\mathcal{B}(\Omega).
\]
We call a set $A\in\mathcal{B}(\Omega)$ is polar if $c(A)=0$ and a property
holds \textquotedblleft quasi-surely" (q.s.) if it holds outside a polar set.
Throughout the paper, we do not distinguish two random variables $X$ and $Y$
if $X=Y$ q.s..

We also define $\mathbb{L}_{G}^{\infty}(\Omega)$ (resp., $\mathbb{L}%
_{G}^{\infty}(\Omega_{T})$) by the completion of $L_{ip}(\Omega)$ (resp.,
$L_{ip}(\Omega_{T})$) under the norm $\Vert\cdot\Vert_{\mathbb{L}_{G}^{\infty
}}:=\inf\{M\geq0,|\cdot|\leq M$ q.s.$\}$. We remark here that a bounded random
variable $\xi\in\mathbb{L}_{G}^{1}(\Omega)$ may not belong to $\mathbb{L}%
_{G}^{\infty}(\Omega)$, which is different from the result in the classical
probability framework.

Next we list the spaces of stochastic processes under $G$-framework on finite
horizon $[0,T]$ for $T>0$:

\begin{itemize}
\item $\mathcal{H}_{G}^{0}(0,T):=\{\eta_{t}\left(  \omega\right)  :=\sum
_{j=0}^{N-1}\xi_{j}\left(  \omega\right)  \mathbf{1}_{[t_{j},t_{j+1})}\left(
t\right)  ,$ $0=t_{0}<\cdots<t_{N}=T$, $N\geq1$, $\xi_{j}\in L_{ip}%
(\Omega_{t_{j}})$ for $j=1,\ldots,N\}$;

\item $\mathcal{H}_{G}^{p}(0,T):=$completion of $\mathcal{H}_{G}^{0}(0,T)$
under the norm $\Vert\cdot\Vert_{\mathcal{H}_{G}^{p}}:=(\mathbb{\hat{E}}%
[(\int_{0}^{T}|\cdot|^{2}ds)^{p/2}])^{1/p}$ for $p\geq1$;

\item $\mathcal{H}_{G}^{p}(0,T;\mathbb{R}^{d}):=$ $\{\eta=(\eta^{i})_{i=1}%
^{d}$, $\eta^{i}\in\mathcal{H}_{G}^{p}(0,T)\}$ for $p\geq1$;

\item $\mathcal{S}_{G}^{0}(0,T):=\{h\left(  t,B_{t_{1}\wedge t},\ldots
,B_{t_{n}\wedge t}\right)  ,$ $t_{1},\ldots,t_{n}\in\lbrack0,T]$, $h\in
C_{b,Lip}(\mathbb{R}^{n+1})\}$;

\item $\mathcal{S}_{G}^{p}(0,T):=$completion of $\mathcal{S}_{G}^{0}(0,T)$
under the norm $\Vert\cdot\Vert_{\mathcal{S}_{G}^{p}}:=|\mathbb{\hat{E}}%
[\sup_{t\in\lbrack0,T]}|\cdot|^{p}]|^{1/p}$ for $p\geq1$.
\end{itemize}

\noindent Having introduced the processes defined on $[0,T]$, we can similarly
consider those on the infinite horizon $[0,\infty)$ and define the following spaces:

\begin{itemize}
\item $\mathcal{H}_{G}^{p}(0,\infty;\mathbb{R}^{d}):=\cap_{T\geq0}%
\mathcal{H}_{G}^{p}(0,T;\mathbb{R}^{d})$;

\item $\mathcal{S}_{G}^{p}(0,\infty):=\cap_{T\geq0}\mathcal{S}_{G}^{p}(0,T)$.
\end{itemize}

Following \cite{LP2011}, we can define $G$-stochastic integrals $\int\eta
d\langle B\rangle^{ij}$ and $\int\eta dB$ for each $\eta\in\mathcal{H}_{G}%
^{1}(0,T)$, and we have the following properties.

\begin{proposition}
\label{BDG}For processes $\eta,\theta\in\mathcal{H}_{G}^{\alpha}%
(0,T;\mathbb{R}^{d})$, $\alpha\geq1$, and bounded random variable $\xi
\in\mathbb{L}_{G}^{1}(\Omega_{t}),0\leq t\leq T$, we have%
\[
\mathbb{\hat{E}}\left[  \int_{0}^{T}\eta_{s}dB_{s}\right]  =0 ~\text{ and }%
\int_{0}^{T}\left(  \xi\eta_{s}+\theta_{s}\right)  dB_{s}=\xi\int_{0}^{T}%
\eta_{s}dB_{s}+\int_{0}^{T}\theta_{s}dB_{s}.
\]
and the Burkholder-Davis-Gundy type inequality holds: for $p>0$,%
\[
c_{G,p}\mathbb{\hat{E}}\left[  \left(  \int_{0}^{T}\left\vert \eta
_{s}\right\vert ^{2}ds\right)  ^{p/2}\right]  \leq\mathbb{\hat{E}}\left[
\sup_{t\in\left[  0,T\right]  }\left\vert \int_{0}^{t}\eta_{s}dB_{s}%
\right\vert ^{p}\right]  \leq C_{G,p}\mathbb{\hat{E}}\left[  \left(  \int
_{0}^{T}\left\vert \eta_{s}\right\vert ^{2}ds\right)  ^{p/2}\right]  ,
\]
where $0<c_{G,p}<C_{G,p}<\infty$ are constants depending on $G$ and $p$.
\end{proposition}

For more details related to $G$-Brownian motion and $G$-It\^{o} calculus, we
refer to \cite[Chapter 3]{PengBOOK}.

\subsection{$G$-BMO martingale generator and $G$-Girsanov's
formula\label{BMOandGirsanov}}

Here we provide the definition of the $G$-BMO martingale generator, introduced
by \cite{HLS2018}, as well as its properties, and the $G$-Girsanov's formula
is also presented in view of $G$-BMO martingale generators. We first recall
the notion of $G$-martingale.

\begin{definition}
A process $M$ is called a $G$-martingale if $M_{t}\in\mathbb{L}_{G}^{1}%
(\Omega_{t})$ for all $t\geq0$, and $\mathbb{\hat{E}}_{s}[M_{t}]=M_{s}$ for
all $0\leq s\leq t<\infty$. It is called a symmetric $G$-martingale if $-M$ is
also a $G$-martingale.
\end{definition}

\begin{definition}
[\cite{HLS2018}]\label{GBMO}A process $\lambda\in\mathcal{H}_{G}%
^{2}(0,T;\mathbb{R}^{d})$ is called a $G$-BMO martingale generator if
\[
\left\Vert \lambda\right\Vert _{BMO_{G}}^{2}:=\sup_{\mathbb{P\in}%
\mathcal{P}_{G}}\left\Vert \lambda\right\Vert _{BMO(\mathbb{P})}^{2}%
:=\sup_{\mathbb{P\in}\mathcal{P}_{G}}\sup_{\tau\in\mathcal{T}_{0}^{T}%
}\left\Vert E_{\tau}^{\mathbb{P}}\left[
%TCIMACRO{\tint _{\tau}^{T}}%
%BeginExpansion
{\textstyle\int_{\tau}^{T}}
%EndExpansion
\left\vert \lambda_{t}\right\vert ^{2}dt\right]  \right\Vert _{\mathbb{L}%
^{\infty}(\mathbb{P})}<\infty,
\]
where $\mathcal{T}_{0}^{T}$ denotes the set of all $[0,T]$-valued $\mathbb{F}%
$-stopping times, and $\Vert\xi\Vert_{\mathbb{L}^{\infty}(\mathbb{P})}%
:=\inf\{M\geq0,|\xi|\leq M$ $\mathbb{P}$-a.s.$\}$.
\end{definition}

For a $G$-BMO martingale generator $\lambda\in\mathcal{H}_{G}^{2}%
(0,T;\mathbb{R}^{d})$, it is obvious that $\left\Vert \lambda\right\Vert
_{BMO(\mathbb{P})}^{2}<\infty$ for all $\mathbb{P}\in\mathcal{P}_{G}$.
Recalling \cite[Theorem 3.1]{KBOOK} and \cite[Lemma 2.1]{PZ2013}, we can find
a $q>1$ such that%
\begin{equation}
\left\Vert \lambda\right\Vert _{BMO_{G}}\leq\Phi\left(  q\right) ~ \text{ with
}~\Phi\left(  q\right)  :=\left(  1+\tfrac{1}{q^{2}}\ln\tfrac{2q-1}{2\left(
q-1\right)  }\right)  ^{1/2}-1. \label{qBMO}%
\end{equation}
Define the stochastic exponential for $\lambda$ as
\[
\mathcal{E}_{t}\left(
%TCIMACRO{\tint }%
%BeginExpansion
{\textstyle\int}
%EndExpansion
\lambda dB\right)  :=\exp\left(  \sum_{i=1}^{d}\int_{0}^{t}\lambda_{s}%
^{i}dB_{s}^{i}-\frac{1}{2}\sum_{i,j=1}^{d}\int_{0}^{t}\lambda_{s}^{i}%
\lambda_{s}^{j}d\left\langle B\right\rangle _{s}^{ij}\right)  ,\text{ }%
t\in\left[  0,T\right]  ,
\]
and we have the following result.

\begin{lemma}
\label{GBMOM}Suppose $\lambda\in\mathcal{H}_{G}^{2}(0,T;\mathbb{R}^{d})$ is a
$G$-BMO martingale generator. Then we have%
\begin{equation}
\mathbb{\hat{E}}\left[  \mathcal{E}_{t}\left(
%TCIMACRO{\tint }%
%BeginExpansion
{\textstyle\int}
%EndExpansion
\lambda dB\right)  ^{q}\right]  <\infty~\text{ and }~\mathcal{E}_{t}\left(
%TCIMACRO{\tint }%
%BeginExpansion
{\textstyle\int}
%EndExpansion
\lambda dB\right)  \in\mathbb{L}_{G}^{1}(\Omega_{t})~\text{ for all }%
t\in\left[  0,T\right]  . \label{RH}%
\end{equation}
Moreover, $\mathcal{E}(%
%TCIMACRO{\tint }%
%BeginExpansion
{\textstyle\int}
%EndExpansion
\lambda dB)$ is a symmetric $G$-martingale.
\end{lemma}

\begin{proof}
The proof is similar to the $1$-dimensional case given in \cite{HLS2018}.
\end{proof}

\medskip Next we introduce the Girsanov's transformation associated with
$G$-BMO martingale generators following the idea in \cite{HLS2018,O2013}.
Based on above lemma, we can define a new $G$-expectation on the space
$L_{ip}(\Omega_{T})$\ with respect to the $G$-BMO martingale generator
$\lambda$ by%
\begin{equation}
\mathbb{\hat{E}}^{\lambda}\left[  X\right]  :=\mathbb{\hat{E}}\left[
\mathcal{E}_{T}\left(
%TCIMACRO{\tint }%
%BeginExpansion
{\textstyle\int}
%EndExpansion
\lambda dB\right)  X\right] ~ \text{ for }X\in L_{ip}(\Omega_{T})\text{,}
\label{GE}%
\end{equation}
and its corresponding conditional $G$-expectation by
\begin{equation}
\mathbb{\hat{E}}_{t}^{\lambda}\left[  X\right]  :=\mathbb{\hat{E}}_{t}\left[
\mathcal{E}_{T}^{t}\left(
%TCIMACRO{\tint }%
%BeginExpansion
{\textstyle\int}
%EndExpansion
\lambda dB\right)  X\right] ~ \text{ for }X\in L_{ip}(\Omega_{T})\text{ and
}t\in\left[  0,T\right]  \label{CE}%
\end{equation}
with%
\[
\mathcal{E}_{T}^{t}\left(
%TCIMACRO{\tint }%
%BeginExpansion
{\textstyle\int}
%EndExpansion
\lambda dB\right)  =\frac{\mathcal{E}_{T}\left(
%TCIMACRO{\tint }%
%BeginExpansion
{\textstyle\int}
%EndExpansion
\lambda dB\right)  }{\mathcal{E}_{t}\left(
%TCIMACRO{\tint }%
%BeginExpansion
{\textstyle\int}
%EndExpansion
\lambda dB\right)  }.
\]
Denote by $\mathbb{L}_{G}^{\lambda,1}(\Omega_{T})$ the completion of
$L_{ip}(\Omega_{T})$ under $\mathbb{\hat{E}}^{\lambda}[|\cdot|]$. Then
$\mathbb{\hat{E}}^{\lambda}$ and $\mathbb{\hat{E}}_{t}^{\lambda}$ can be
extended to the mappings on $\mathbb{L}_{G}^{\lambda,1}(\Omega_{T})$ which are
still denoted by $\mathbb{\hat{E}}^{\lambda}$ and $\mathbb{\hat{E}}%
_{t}^{\lambda}$. Furthermore, (\ref{GE}) and (\ref{CE}) still hold on
$\mathbb{L}_{G}^{\lambda,1}(\Omega_{T})$. From Lemma \ref{GBMOM}, we directly
have the following Girsanov's formula in $G$-framework by \cite[Theorem 5.3]{O2013}.

\begin{theorem}
\label{Girsanov}For a $G$-BMO martingale generator $\lambda\in\mathcal{H}%
_{G}^{2}(0,T;\mathbb{R}^{d})$, $B^{\lambda}=(B^{\lambda,i})_{i=1}^{d}$ with
each component defined by%
\begin{equation}
B_{t}^{\lambda,i}:=B_{t}^{i}-\sum_{j=1}^{d}\int_{0}^{t}\lambda_{s}%
^{j}d\left\langle B\right\rangle _{s}^{ij}\text{, }t\in\left[  0,T\right]  ,
\label{GirBM}%
\end{equation}
is a $d$-dimensional $G$-Brownian motion on the space $(\Omega_{T}%
,\mathbb{L}_{G}^{\lambda,1}(\Omega_{T}),\mathbb{\hat{E}}^{\lambda})$.
\end{theorem}

\begin{remark}
\label{ReBMO}We can check the following integrability results for $G$-BMO
martingale generators: for a $G$-BMO martingale generator $\lambda
\in\mathcal{H}_{G}^{2}(0,T;\mathbb{R}^{d})$ with $\left\Vert \lambda
\right\Vert _{BMO_{G}}\leq\Phi(q)$ for a $q>1$,

\begin{enumerate}
\item[(i)] $\left\Vert \lambda\right\Vert _{\mathcal{H}_{G}^{p}}<\infty$ for
any $p\geq1$;

\item[(ii)] if $X\in\mathbb{L}_{G}^{p}(\Omega_{T})$ for some $p>\frac{q}{q-1}%
$, then we have $X\in\mathbb{L}_{G}^{\lambda,1}(\Omega_{T})$;

\item[(iii)] for any $G$-BMO martingale generator $Z\in\mathcal{H}_{G}%
^{2}(0,T;\mathbb{R}^{d})$ and $p\geq1$,%
\begin{align*}
\mathbb{\hat{E}}^{\lambda}\left[  \left(  \int_{0}^{T}\left\vert
Z_{s}\right\vert ^{2}ds\right)  ^{p/2}\right]   &  =\mathbb{\hat{E}}\left[
\mathcal{E}_{T}\left(
%TCIMACRO{\tint }%
%BeginExpansion
{\textstyle\int}
%EndExpansion
\lambda dB\right)  \left(  \int_{0}^{T}\left\vert Z_{s}\right\vert
^{2}ds\right)  ^{p/2}\right] \\
&  \leq\mathbb{\hat{E}}\left[  \mathcal{E}_{T}\left(
%TCIMACRO{\tint }%
%BeginExpansion
{\textstyle\int}
%EndExpansion
\lambda dB\right)  ^{q}\right]  ^{1/q}\mathbb{\hat{E}}\left[  \left(  \int
_{0}^{T}\left\vert Z_{s}\right\vert ^{2}ds\right)  ^{pq/2\left(  q-1\right)
}\right]  ^{\left(  q-1\right)  /q}\\
&  <\infty,
\end{align*}
and, moreover,
\[
\left(  \int_{0}^{t}Z_{s}^{i}ds\right)  ^{p}\in\mathbb{L}_{G}^{1}(\Omega
_{t})\text{, }\left(  \int_{0}^{t}Z_{s}^{i}ds\right)  ^{p}\in\mathbb{L}%
_{G}^{\lambda,1}(\Omega_{t}),~\text{ }i=1,\ldots,d.
\]

\end{enumerate}
\end{remark}

\subsection{Quadratic $G$-BSDEs on finite horizon}

In order to prove the solvability of quadratic $G$-BSDEs on the infinite
horizon, we will make use of the results for finite horizon case. For reader's
convenience, we present the well-posedness and the estimates for quadratic
$G$-BSDEs on finite horizon obtained in \cite{HLS2018}.

For a given $T>0$, consider the following quadratic $G$-BSDE: for $0\leq t\leq
T$,
\begin{equation}
Y_{t}=\xi+\int_{t}^{T}f\left(  s,Y_{s},Z_{s}\right)  ds+\sum_{i,j=1}^{d}%
\int_{t}^{T}g^{ij}\left(  s,Y_{s},Z_{s}\right)  d\left\langle B\right\rangle
_{s}^{ij}-\int_{t}^{T}Z_{s}dB_{s}-\left(  K_{T}-K_{t}\right)  , \label{fQBSDE}%
\end{equation}
where $\xi\in\mathbb{L}_{G}^{\infty}(\Omega_{T})$ and the generators
$f,g^{ij}=g^{ji}:[0,T]\times\Omega\times\mathbb{R}\times\mathbb{R}%
^{d}\rightarrow\mathbb{R}$ are supposed to satisfy the following conditions:

\begin{condition}
\label{Hf}For $h=f,g^{ij}$, $i,j=1,\ldots,d$,

\begin{enumerate}
\item[(i)] $|h(t,\omega,0,0)|\leq M_{0}$ for all $(t,\omega)\in\lbrack
0,T]\times\Omega$,

\item[(ii)] $h(t,\omega,y,z)$ is uniformly continuous in $(t,\omega)$ and the
modulus of continuity $w^{h}$ is independent of $(y,z)$: for each $\left(
y,z\right)  \in\mathbb{R}\times\mathbb{R}^{d}$,
\[
|h\left(  t,\omega,y,z\right)  -h\left(  t^{\prime},\omega^{\prime
},y,z\right)  |\leq w^{h}\left(  \left\vert t-t^{\prime}\right\vert
+\left\Vert \omega-\omega^{\prime}\right\Vert _{\infty}\right)  \text{,}%
\]
where $\Vert\omega-\omega^{\prime}\Vert_{\infty}:=\max_{t\in\lbrack
0,T]}|\omega_{t}-\omega_{t}^{\prime}|$;

\item[(iii)] $h$ is uniformly Lipschitz continuous in $y$ and uniformly
locally Lipschitz continuous in $z$: there exist positive constants $L_{y}$
and $L_{z}$ such that, for each $(t,\omega)\in\lbrack0,T]\times\Omega$,%
\[
\left\vert h\left(  t,\omega,y,z\right)  -h\left(  t,\omega,y^{\prime
},z^{\prime}\right)  \right\vert \leq L_{y}\left\vert y-y^{\prime}\right\vert
+L_{z}\left(  1+\left\vert z\right\vert +\left\vert z^{\prime}\right\vert
\right)  \left\vert z-z^{\prime}\right\vert .
\]

\end{enumerate}
\end{condition}

According to \cite[Remark 2.15]{HLS2018}, we know that, for $h=f,g^{ij}$ and
each $p\geq2$, $h(\cdot,\cdot,y,z)\in\mathcal{H}_{G}^{p}(0,T)$ and
$h(\cdot,\cdot,Y,Z)\in\mathcal{H}_{G}^{p}(0,T)$ if $Y\in\mathcal{S}_{G}%
^{p}(0,T)$ and $Z\in\mathcal{H}_{G}^{2p}(0,T;\mathbb{R}^{d})$.

\begin{definition}
We call the triplet $(Y,Z,K)$ a solution to $G$-BSDE (\ref{fQBSDE}) if

\begin{enumerate}
\item[(i)] $(Y,Z,K)\in\mathfrak{S}_{G}^{2}(0,T)$, where $\mathfrak{S}_{G}%
^{2}(0,T):=\{(Y,Z,K):Y\in\mathcal{S}_{G}^{2}(0,T),Z\in\mathcal{H}_{G}%
^{2}(0,T;\mathbb{R}^{d})$ and $K$ is a nonincreasing $G$-martingale with
$K_{0}=0$ and $K_{T}\in\mathbb{L}_{G}^{2}(\Omega_{T})\}\,$;

\item[(ii)] it satisfies the equation (\ref{fQBSDE}) for all $t\in\lbrack0,T]$.
\end{enumerate}
\end{definition}

\begin{theorem}
[\cite{HLS2018}]\label{QGBSDE}Let Assumption \ref{Hf} hold. Then the
quadratic $G$-BSDE (\ref{fQBSDE}) on finite horizon $[0,T]$ admits a unique
solution $(Y,Z,K)\in\mathfrak{S}_{G}^{2}(0,T)$ such that $Y$ is bounded.
\end{theorem}

\begin{remark}
\label{1101}If $(Y,Z,K)\in\mathfrak{S}_{G}^{2}(0,T)$ is the solution to
$G$-BSDE (\ref{fQBSDE}) with $Y$ being bounded, then there exist positive
constants $\hat{C}$ and $\tilde{C}_{p}$ such that%
\begin{equation}
\left\Vert Z\right\Vert _{BMO_{G}}\leq\hat{C} \label{f1}%
\end{equation}
and%
\begin{equation}
\mathbb{\hat{E}}\left[  \left\vert K_{t}\right\vert ^{p}\right]  \leq\tilde
{C}_{p}~\text{ for }0\leq t\leq T\text{ and }p\geq1. \label{f2}%
\end{equation}

\end{remark}

\section{Explicit solutions of linear $G$-BSDEs with unbounded
coefficients\label{Explicit solution}}

This section is devoted to the explicit solutions of linear $G$-BSDEs on
finite horizon with unbounded coefficients of $Z$-components. In fact, linear
equations play an important role of giving the estimates of the difference of
two BSDEs in the study of infinite horizon BSDEs (see, e.g.,
\cite{BH1998,HSW2022,HW2018}).

For a given $T>0$, we consider the following linear $G$-BSDE: for $0\leq t\leq
T$,%
\begin{equation}
Y_{t}=\xi+\int_{t}^{T}f_{s}ds+\sum_{i,j=1}^{d}\int_{t}^{T}g_{s}^{ij}%
d\left\langle B\right\rangle _{s}^{ij}-\int_{t}^{T}Z_{s}dB_{s}-\left(
K_{T}-K_{t}\right)  , \label{LE}%
\end{equation}
where%
\begin{equation}
\xi\in\mathbb{L}_{G}^{\infty}(\Omega_{T})\text{, \ }f_{s}=a_{s}Y_{s}%
+b_{s}Z_{s}+m_{s}\text{, \ }g_{s}^{ij}=c_{s}^{ij}Y_{s}+d_{s}^{ij}Z_{s}%
+n_{s}^{ij}, \label{gen}%
\end{equation}
and
\begin{equation}
a,c^{ij}=c^{ji},m,n^{ij}=n^{ji}\in\mathcal{H}_{G}^{2}(0,T)\text{, \ }%
b,d^{ij}=d^{ji}\in\mathcal{H}_{G}^{2}(0,T;\mathbb{R}^{d}). \label{coe}%
\end{equation}
In the rest of this section, we always assume the following conditions on the coefficients:

\begin{enumerate}
\item[(i)] $a$ and $c^{ij}$ are uniformly bounded;

\item[(ii)] $b$ and $d^{ij}$ are $d$-dimensional $G$-BMO martingale generators
and there exists a $q>1$ such that $\left\Vert \Lambda\right\Vert _{BMO_{G}%
}\leq\Phi(q)$ for $\Lambda$ defined by (\ref{lamda}) below;

\item[(iii)] $m,n^{ij}\in\mathcal{H}_{G}^{p}(0,T)$ for some $p>\frac{q}{q-1}.$
\end{enumerate}

We will gradually derive the representation for the explicit solutions to
linear $G$-BSDE (\ref{LE}) from the cases $d=1$ and $d=2$ to the general case
in order to illustrate the idea of construction.

\subsection{Case $d=1$}

We start with the case $d=1$ to reveal the role of $G$-Girsanov's formula in
solving linear $G$-BSDE as well as the reason why we need to extend the
dimension of the $G$-expectation space.

Consider a $1$-dimensional $G$-expectation space with $G$ given by%
\[
G\left(  a\right)  :=\frac{1}{2}\left(  a^{+}-\sigma^{2}a^{-}\right)
,\text{\thinspace}\sigma\in(0,1],
\]
and the linear $G$-BSDE (\ref{LE}) is of the form
\begin{equation}
Y_{t}=\xi+\int_{t}^{T}\left(  a_{s}Y_{s}+b_{s}Z_{s}+m_{s}\right)  ds+\int
_{t}^{T}\left(  c_{s}Y_{s}+d_{s}Z_{s}+n_{s}\right)  d\left\langle
B\right\rangle _{s}-\int_{t}^{T}Z_{s}dB_{s}-\left(  K_{T}-K_{t}\right)  .
\label{L1}%
\end{equation}
Let $X$ be the solution of the following linear $G$-SDE:%
\[
dX_{t}=a_{t}X_{t}dt+c_{t}X_{t}d\left\langle B\right\rangle _{t}\text{, }%
~X_{0}=1\text{,}%
\]
and it can be represented as
\begin{equation}
X_{t}=\exp\left(  \int_{0}^{t}a_{s}ds+\int_{0}^{t}c_{s}d\left\langle
B\right\rangle _{t}\right)  . \label{X1}%
\end{equation}
Applying $G$-It\^{o}'s formula to $XY$ implies that%
\begin{equation}
X_{t}Y_{t}+\int_{t}^{T}X_{s}dK_{s}+\int_{t}^{T}X_{s}Z_{s}\left(  dB_{s}%
-b_{s}ds-d_{s}d\langle B\rangle_{s}\right)  =X_{T}\xi+\int_{t}^{T}m_{s}%
X_{s}ds+\int_{t}^{T}n_{s}X_{s}d\langle B\rangle_{s}. \label{Q1801}%
\end{equation}
In order to get the representation of $Y$, we want to eliminate the last two
terms of the left hand side via taking conditional expectation. However,
according to Theorem \ref{Girsanov}, the original space
$(\Omega_{T},\mathbb{L}_{G}^{1}(\Omega_{T}),\mathbb{\hat{E}})$ is not big
enough to transform $B_{t}-\int_{0}^{t}b_{s}ds-\int_{0}^{t}d_{s}d\langle
B\rangle_{s}$ into a $G$-Brownian motion due to the extra $ds$ term. Thus, the
main task is to define an auxiliary extended $\tilde{G}$-expectation space on
which $B_{t}-\int_{0}^{t}b_{s}ds-\int_{0}^{t}d_{s}d\langle B\rangle_{s}$ is a
$\tilde{G}$-Brownian motion and $\int_{0}^{t}X_{s}dK_{s}$ is a nonincreasing
$\tilde{G}$-martingale.

Similarly to \cite{HJPS2014b}, define%
\begin{equation}
\tilde{G}\left(  A\right)  :=\frac{1}{2}\sup_{\tilde{\theta}\in\tilde{\Theta}%
}\operatorname*{tr}[\tilde{\theta}\tilde{\theta}^{\top}A]~\text{ for }%
A\in\mathbb{S}\left(  2\right)  , \label{G2}%
\end{equation}
with%
\[
\tilde{\Theta}:=\left\{  \tilde{\theta}=%
\begin{pmatrix}
\theta & 0\\
\theta^{-1} & 1
\end{pmatrix}
,\theta\in\left[  \sigma,1\right]  \right\}  ,
\]
and, in turn, we can construct a $\tilde{G}$-expectation space $(\tilde
{\Omega}_{T}=C_{0}([0,T];\mathbb{R}^{2}),\mathbb{L}_{\tilde{G}}^{1}%
(\tilde{\Omega}_{T}),\mathbb{\hat{E}}^{\tilde{G}})$ with $2$-dimensional
$\tilde{G}$-Brownian motion denoted by $\tilde{B}=(B,\dot{B})^{\top}$. For
$\tilde{\theta}\in\tilde{\Theta}$ with its corresponding $\theta\in
\lbrack\sigma,1]$, we have
\begin{equation}
\tilde{\theta}\tilde{\theta}^{\top}=%
\begin{pmatrix}
\theta^{2} & 1\\
1 & \left(  \theta^{-1}\right)  ^{2}+1
\end{pmatrix}
\text{.} \label{M1}%
\end{equation}
It then follows by \cite[Corollary 3.5.8]{PengBOOK} that $\langle B,\dot
{B}\rangle_{t}=t$. Moreover, we have $\mathbb{\hat{E}}^{\tilde{G}%
}=\mathbb{\hat{E}}$ on $\mathbb{L}_{G}^{1}(\Omega_{T})$.

Let $\Lambda:=(d,b)^{\top}\in\mathcal{H}_{G}^{2}(0,T;\mathbb{R}^{2})$, which
is a $G$-BMO martingale generator and also a $\tilde{G}$-BMO martingale
generator. Then, we can define a $\tilde{G}$-expectation space $(\tilde
{\Omega}_{T},\mathbb{L}_{\tilde{G}}^{\Lambda,1}(\tilde{\Omega}_{T}%
),\mathbb{\hat{E}}^{\tilde{G},\Lambda})$ according to the procedure in Section
\ref{BMOandGirsanov} with%
\[
\mathbb{\hat{E}}^{\tilde{G},\Lambda}\left[  X\right]  =\mathbb{\hat{E}%
}^{\tilde{G}}\left[  \mathcal{E}_{T}\left(
%TCIMACRO{\tint }%
%BeginExpansion
{\textstyle\int}
%EndExpansion
\Lambda d\tilde{B}\right)  X\right] ~ \text{ and }~\mathbb{\hat{E}}_{t}%
^{\tilde{G},\Lambda}\left[  X\right]  =\mathbb{\hat{E}}_{t}^{\tilde{G}}\left[
\mathcal{E}_{T}^{t}\left(
%TCIMACRO{\tint }%
%BeginExpansion
{\textstyle\int}
%EndExpansion
\Lambda d\tilde{B}\right)  X\right]  ~\text{ for }X\in\mathbb{L}_{\tilde{G}%
}^{\Lambda,1}(\tilde{\Omega}_{T})\text{.}%
\]
By Theorem \ref{Girsanov}, we know that
\begin{equation}
\tilde{B}_{t}^{\Lambda,1}:=B_{t}-\int_{0}^{t}b_{s}ds-\int_{0}^{t}d_{s}d\langle
B\rangle_{s},\text{ }t\in\lbrack0,T], \label{GBM1}%
\end{equation}
is a $\tilde{G}$-Brownian motion under $\mathbb{\hat{E}}^{\tilde{G},\Lambda}$.
Now we are in a position to give the representation result for linear BSDE
driven by $1$-dimensional $G$-Brownian motion.

\begin{proposition}
\label{1dimProp}Suppose $(Y,Z,K)\in\mathfrak{S}_{G}^{2}(0,T)$ is the solution
of linear $G$-BSDE (\ref{LE}) satisfying the following conditions:

\begin{enumerate}
\item[(i)] $Z$ is a $1$-dimensional $G$-BMO martingale generator,

\item[(ii)] $K_{T}\in\mathbb{L}_{G}^{p}(\Omega_{T})$ for some $p>\frac{q}%
{q-1}$.
\end{enumerate}

\noindent\noindent\noindent Then we have%
\[
Y_{t}=\left(  X_{t}\right)  ^{-1}\mathbb{\hat{E}}_{t}^{\tilde{G},\Lambda
}\left[  X_{T}\xi+\int_{t}^{T}m_{s}X_{s}ds+\int_{t}^{T}n_{s}X_{s}d\left\langle
B\right\rangle _{s}\right]
\]
in the $2$-dimensional $\tilde{G}$-expectation space $(\tilde{\Omega}%
_{T},\mathbb{L}_{\tilde{G}}^{\Lambda,1}(\tilde{\Omega}_{T}),\mathbb{\hat{E}%
}^{\tilde{G},\Lambda})$ with $\tilde{G}$ in (\ref{G2}), $\Lambda:=(d,b)^{\top
}$ and $X$ given by (\ref{X1}).
\end{proposition}

\begin{proof}
We have shown that $\tilde{B}^{\Lambda,1}$ in (\ref{GBM1}) is a $\tilde{G}%
$-Brownian motion under $\mathbb{\hat{E}}^{\tilde{G},\Lambda}$. Noting that
$Z$ is a $\tilde{G}$-BMO martingale generator and $X$ is uniformly bounded, we
can check by Remark \ref{ReBMO} that $\mathbb{\hat{E}}^{\tilde{G},\Lambda
}[(\int_{t}^{T}|X_{s}Z_{s}|^{2}ds)^{p/2}]<\infty$ for any $p\geq1$, and hence
$\mathbb{\hat{E}}_{t}^{\tilde{G},\Lambda}[\int_{t}^{T}X_{s}Z_{s}d\tilde{B}%
_{s}^{\Lambda,1}]=0$. On the other hand, since $\mathbb{\hat{E}}^{\tilde{G}}$
coincides with $\mathbb{\hat{E}}$ on $\mathbb{L}_{G}^{1}(\Omega_{T})$, $K$ is
a nonincreasing $\tilde{G}$-martingale under $\mathbb{\hat{E}}^{\tilde{G}}$.
Applying \cite[Lemma 3.4]{HLS2018} and \cite[Lemma 3.4]{HJPS2014a}, we obtain
that $\int_{0}^{t}X_{s}dK_{s}$ is a nonincreasing $\tilde{G}$-martingale under
$\mathbb{\hat{E}}^{\tilde{G},\Lambda}$. Finally, taking $\mathbb{\hat{E}}%
_{t}^{\tilde{G},\Lambda}$ on both sides of (\ref{Q1801}) yields that%
\[
X_{t}Y_{t}=\mathbb{\hat{E}}_{t}^{\tilde{G},\Lambda}\left[  X_{T}Y_{T}+\int
_{t}^{T}m_{s}X_{s}ds+\int_{t}^{T}n_{s}X_{s}d\left\langle B\right\rangle
_{s}\right]  ,
\]
which gives the desired result.
\end{proof}

\subsection{Case $d=2$}

Based on the previous discussion, when the dimension of the space is greater
than $1$, we still need to extend the space to derive the explicit solution of
linear $G$-BSDE (\ref{LE}). However, the construction of the extended space is
more complicated compared with the case $d=1$. Here we give the detailed
procedure for case $d=2$ for reader's convenience.

Consider a $2$-dimensional $G$-expectation space with $G:\mathbb{S}%
(2)\rightarrow\mathbb{R}$ represented by $\Theta\subset\mathbb{R}^{2\times2}$
as in (\ref{G}). We write the linear $G$-BSDE (\ref{LE}) in detail:
\begin{align}
Y_{t}=  &  \xi+\int_{t}^{T}\left(  a_{s}Y_{s}+b_{s}^{1}Z_{s}^{1}+b_{s}%
^{2}Z_{s}^{2}+m_{s}\right)  ds+\int_{t}^{T}\left(  c_{s}^{11}Y_{s}%
+d_{s}^{11,1}Z_{s}^{1}+d_{s}^{11,2}Z_{s}^{2}+n_{s}^{11}\right)  d\left\langle
B\right\rangle _{s}^{11}\label{2dimLG}\\
&  +2\int_{t}^{T}\left(  c_{s}^{12}Y_{s}+d_{s}^{12,1}Z_{s}^{1}+d_{s}%
^{12,2}Z_{s}^{2}+n_{s}^{12}\right)  d\left\langle B\right\rangle _{s}%
^{12}\nonumber\\
&  +\int_{t}^{T}\left(  c_{s}^{22}Y_{s}+d_{s}^{22,1}Z_{s}^{1}+d_{s}%
^{22,2}Z_{s}^{2}+n_{s}^{22}\right)  d\left\langle B\right\rangle _{s}%
^{22}-\int_{t}^{T}\left(  Z_{s}^{1}dB_{s}^{1}+Z_{s}^{2}dB_{s}^{2}\right)
-\left(  K_{T}-K_{t}\right)  .\nonumber
\end{align}
Let $X$ satisfy $G$-SDE%
\[
dX_{t}=a_{t}X_{t}dt+c_{t}^{11}X_{t}d\left\langle B\right\rangle _{t}%
^{11}+2c_{t}^{12}X_{t}d\left\langle B\right\rangle _{t}^{12}+c_{t}^{22}%
X_{t}d\left\langle B\right\rangle _{t}^{22},\text{ }X_{0}=1,
\]
and we have%
\begin{equation}
X_{t}=\exp\left(  \int_{0}^{t}a_{s}ds+\int_{0}^{t}\left(  c_{s}^{11}%
d\left\langle B\right\rangle _{s}^{11}+2c_{s}^{12}d\left\langle B\right\rangle
_{s}^{12}+c_{s}^{22}d\left\langle B\right\rangle _{s}^{22}\right)  \right)  .
\label{X2}%
\end{equation}
Similarly to the case $d=1$, using $G$-It\^{o}'s formula gives that%
\begin{align*}
&  X_{t}Y_{t}+\int_{t}^{T}X_{s}dK_{s}+\int_{t}^{T}X_{s}Z_{s}^{1}\left(
dB_{s}^{1}-b_{s}^{1}ds-d_{s}^{11,1}d\left\langle B\right\rangle _{s}%
^{11}-d_{s}^{22,1}d\left\langle B\right\rangle _{s}^{22}-2d_{s}^{12,1}%
d\left\langle B\right\rangle _{s}^{12}\right) \\
&  +\int_{t}^{T}X_{s}Z_{s}^{2}\left(  dB_{s}^{2}-b_{s}^{2}ds-d_{s}%
^{22,2}d\left\langle B\right\rangle _{s}^{22}-d_{s}^{11,2}d\left\langle
B\right\rangle _{s}^{11}-2d_{s}^{12,2}d\left\langle B\right\rangle _{s}%
^{12}\right) \\
=  &  X_{T}\xi+\int_{t}^{T}m_{s}X_{s}ds+\int_{t}^{T}\left(  n_{s}^{11}%
X_{s}d\left\langle B\right\rangle _{s}^{11}+2n_{s}^{12}X_{s}d\left\langle
B\right\rangle _{s}^{12}+n_{s}^{22}X_{s}d\left\langle B\right\rangle _{s}%
^{22}\right)  .
\end{align*}
So we need to construct an auxiliary $\tilde{G}$-expectation space with enough
dimension such that
\[
B_{t}^{1}-\int_{0}^{t}b_{s}^{1}ds-\int_{0}^{t}d_{s}^{11,1}d\langle
B\rangle_{s}^{11}-\int_{0}^{t}d_{s}^{22,1}d\langle B\rangle_{s}^{22}-\int
_{0}^{t}2d_{s}^{12,1}d\langle B\rangle_{s}^{12},\text{ }t\in\lbrack0,T],
\]
and%
\[
B_{t}^{2}-\int_{0}^{t}b_{s}^{2}ds-\int_{0}^{t}d_{s}^{22,2}d\langle
B\rangle_{s}^{22}-\int_{0}^{t}d_{s}^{11,2}d\langle B\rangle_{s}^{11}-\int
_{0}^{t}2d_{s}^{12,2}d\langle B\rangle_{s}^{12},\text{ }t\in\lbrack0,T],
\]
can be transformed into a $\tilde{G}$-Brownian motion via Girsanov's formula.
To this end, we define
\begin{equation}
\tilde{G}\left(  A\right)  :=\frac{1}{2}\sup_{\tilde{\theta}\in\tilde{\Theta}%
}\operatorname*{tr}[\tilde{\theta}\tilde{\theta}^{\top}A]~\text{ for }%
A\in\mathbb{S}(8), \label{1301}%
\end{equation}
where $\tilde{\Theta}\subset\mathbb{R}^{8\times8}$ is the set of all
$8\times8$ matrices of the form%
\[
\tilde{\theta}=\left(
\begin{array}
[c]{cccc}%
\theta & O_{2} & O_{2} & O_{2}\\
\theta^{-1,\top} & I_{2} & O_{2} & O_{2}\\%
\begin{pmatrix}
0 & \gamma^{11}\\
\gamma^{22} & 0
\end{pmatrix}
\theta^{-1,\top} & O_{2} & I_{2} & O_{2}\\
\gamma^{12}\theta^{-1,\top} & O_{2} & O_{2} & I_{2}%
\end{array}
\right) ~ \text{ with }~%
\begin{pmatrix}
\gamma^{11} & \gamma^{12}\\
\gamma^{12} & \gamma^{22}%
\end{pmatrix}
:=\theta\theta^{\top}~\text{ for }\theta\in\Theta\subset\mathbb{R}^{2\times2}.
\]
We can, in turn, construct a $\tilde{G}$-expectation space $(\tilde{\Omega
}_{T}=C_{0}([0,T];\mathbb{R}^{8}),\mathbb{L}_{\tilde{G}}^{1}(\tilde{\Omega
}_{T}),\mathbb{\hat{E}}^{\tilde{G}})$ and let $\tilde{B}=(\tilde{B}^{l}%
)_{l=1}^{8}$ be the corresponding $\tilde{G}$-Brownian motion. We remark that
$\mathbb{\hat{E}}^{\tilde{G}}=\mathbb{\hat{E}}$ on $\mathbb{L}_{G}^{1}%
(\Omega_{T})$. For a clearer understanding in the following, we will make use
of the notations:
\begin{equation}
B^{j}:=\tilde{B}^{j},~\text{ }\dot{B}^{j}:=\tilde{B}^{2+j},~\text{ }\hat{B}%
^{j}:=\tilde{B}^{4+j}~\text{ and }~\bar{B}^{j}:=\tilde{B}^{6+j},~\text{ for
}j=1,2\text{.} \label{N2}%
\end{equation}

\begin{remark}
\label{2dim}In parallel with (\ref{M1}), it is easy to see that, for any
$\theta\in\Theta$ and its corresponding $\tilde{\theta}\in\tilde{\Theta}$,%
\begin{equation}
\tilde{\theta}\tilde{\theta}^{\top}=\left(
\begin{array}
[c]{cccc}%
\begin{array}
[c]{cc}%
\gamma^{11} & \gamma^{12}\\
\gamma^{12} & \gamma^{22}%
\end{array}
&
\begin{array}
[c]{cc}%
1\text{ } & \text{ }0\\
0\text{ } & \text{ }1
\end{array}
&
\begin{array}
[c]{cc}%
0 & \gamma^{22}\\
\gamma^{11} & 0
\end{array}
&
\begin{array}
[c]{cc}%
\gamma^{12} & 0\\
0 & \gamma^{12}%
\end{array}
\\%
\begin{array}
[c]{cc}%
1\text{ } & \text{ }0\\
0\text{ } & \text{ }1
\end{array}
&  &  & \\%
\begin{array}
[c]{cc}%
0 & \gamma^{11}\\
\gamma^{22} & 0
\end{array}
&  & RR^{\top} & \\%
\begin{array}
[c]{cc}%
\gamma^{12} & 0\\
0 & \gamma^{12}%
\end{array}
&  &  &
\end{array}
\right)  \label{Q1802}%
\end{equation}
with%
\[
R:=\left(
\begin{array}
[c]{cccc}%
\theta^{-1,\top} & I_{2} & O_{2} & O_{2}\\%
\begin{pmatrix}
0 & \gamma^{11}\\
\gamma^{22} & 0
\end{pmatrix}
\theta^{-1,\top} & O_{2} & I_{2} & O_{2}\\
\gamma^{12}\theta^{-1,\top} & O_{2} & O_{2} & I_{2}%
\end{array}
\right)  \in\mathbb{R}^{6\times8}~\text{ and }~%
\begin{pmatrix}
\gamma^{11} & \gamma^{12}\\
\gamma^{12} & \gamma^{22}%
\end{pmatrix}
:=\theta\theta^{\top}.
\]
We claim that, for $i,j=1,2,$
\[
\langle B^{i},\dot{B}^{j}\rangle_{t}=\left\{
\begin{array}
[c]{cc}%
t, & \text{if }i=j,\\
0, & \text{if }i\neq j.
\end{array}
\right.  \text{ \ }\langle B^{i},\hat{B}^{j}\rangle_{t}=\left\{
\begin{array}
[c]{ll}%
0, & \text{if }i=j,\\
\langle B\rangle_{t}^{jj}, & \text{if }i\neq j.
\end{array}
\right.  \text{ \ }\langle B^{i},\bar{B}^{j}\rangle_{t}=\left\{
\begin{array}
[c]{ll}%
\langle B\rangle_{t}^{12}, & \text{if }i=j,\\
0, & \text{if }i\neq j.
\end{array}
\right.
\]
In fact, applying \cite[Proposition 3.5.7]{PengBOOK} to $\varphi
(x)=|x^{11}-x^{25}|,x=(x^{ij})_{i,j=1}^{8}\in\mathbb{S}(8)$, we can easily
check by (\ref{N2}) and (\ref{Q1802}) that%
\[
\mathbb{\hat{E}}^{\tilde{G}}\left[  \left\vert \langle B\rangle_{t}%
^{11}-\langle B^{2},\hat{B}^{1}\rangle_{t}\right\vert \right]  =\mathbb{\hat
{E}}^{\tilde{G}}\left[  \left\vert \langle\tilde{B}\rangle_{t}^{11}%
-\langle\tilde{B}\rangle_{t}^{25}\right\vert \right]  =\sup_{x\in\tilde
{\theta}\tilde{\theta}^{\top}}\left\vert x^{11}-x^{25}\right\vert =0,
\]
which implies that $\langle B\rangle_{t}^{11}=\langle B^{2},\hat{B}^{1}%
\rangle_{t}$, and we can verify other equalities similarly.
\end{remark}

Corresponding to each $\tilde{B}^{l},l=1,\ldots,8,$ we choose an
$8$-dimensional process $\Lambda=(\Lambda^{l})_{l=1}^{8}$ using the similar
notation as in (\ref{N2}): for $j=1,2$ and $i\neq j$,%
\begin{equation}
\Lambda^{j}:=\lambda^{j}:=d^{jj,j},\text{ \ \ }\Lambda^{2+j}:=\dot{\lambda
}^{j}:=b^{j},\text{ \ \ }\Lambda^{4+j}:=\hat{\lambda}^{j}:=d^{jj,i},\text{
\ \ }\Lambda^{6+j}:=\bar{\lambda}^{j}:=2d^{ji,j}-d^{ii,i}. \label{2dimlamda}%
\end{equation}
Noting that $\Lambda$ is a $\tilde{G}$-BMO martingale generator, we know from
Theorem \ref{Girsanov} that
\[
\tilde{B}_{t}^{\Lambda,1}:=\tilde{B}_{t}^{1}-\sum_{j=1}^{8}\int_{0}^{t}%
\Lambda_{s}^{j}d\langle\tilde{B}^{1},\tilde{B}^{j}\rangle_{t}~\text{ and
}~\tilde{B}_{t}^{\Lambda,2}:=\tilde{B}_{t}^{2}-\sum_{j=1}^{8}\int_{0}%
^{t}\Lambda_{s}^{j}d\langle\tilde{B}^{2},\tilde{B}^{j}\rangle_{t},\text{ }%
t\in\lbrack0,T],
\]
are the first two components of the $\tilde{G}$-Brownian motion on the space
$(\tilde{\Omega}_{T},\mathbb{L}_{\tilde{G}}^{\Lambda,1}(\tilde{\Omega}%
_{T}),\mathbb{\hat{E}}^{\tilde{G},\Lambda})$, where%
\[
\mathbb{\hat{E}}^{\tilde{G},\Lambda}\left[  X\right]  =\mathbb{\hat{E}%
}^{\tilde{G}}\left[  \mathcal{E}_{T}\left(
%TCIMACRO{\tint }%
%BeginExpansion
{\textstyle\int}
%EndExpansion
\Lambda d\tilde{B}\right)  X\right] ~ \text{ and }~\mathbb{\hat{E}}_{t}%
^{\tilde{G},\Lambda}\left[  X\right]  =\mathbb{\hat{E}}_{t}^{\tilde{G}}\left[
\mathcal{E}_{T}^{t}\left(
%TCIMACRO{\tint }%
%BeginExpansion
{\textstyle\int}
%EndExpansion
\Lambda d\tilde{B}\right)  X\right] ~ \text{ for }X\in\mathbb{L}_{\tilde{G}%
}^{\Lambda,1}(\tilde{\Omega}_{T})\text{.}%
\]
By Remark \ref{2dim} and (\ref{2dimlamda}), we deduce that%
\begin{align*}
\tilde{B}_{t}^{\Lambda,1}=  &  B_{t}^{1}-\int_{0}^{t}d_{s}^{11,1}d\langle
B\rangle_{s}^{11}-\int_{0}^{t}d_{s}^{22,2}d\langle B\rangle_{s}^{12}-\int
_{0}^{t}b_{s}^{1}d\langle B^{1},\dot{B}^{1}\rangle_{s}-\int_{0}^{t}b_{s}%
^{2}d\langle B^{1},\dot{B}^{2}\rangle_{s}\\
&  -\int_{0}^{t}d_{s}^{11,2}d\langle B^{1},\hat{B}^{1}\rangle_{s}-\int_{0}%
^{t}d_{s}^{22,1}d\langle B^{1},\hat{B}^{2}\rangle_{s}\\
&  -\int_{0}^{t}\left(  2d_{s}^{12,1}-d_{s}^{22,2}\right)  d\langle B^{1}%
,\bar{B}^{2}\rangle_{s}-\int_{0}^{t}\left(  2d_{s}^{21,2}-d_{s}^{11,1}\right)
d\langle B^{1},\bar{B}^{2}\rangle_{s}\\
=  &  B_{t}^{1}-\int_{0}^{t}d_{s}^{11,1}d\langle B\rangle_{s}^{11}-\int
_{0}^{t}b_{s}^{1}ds-\int_{0}^{t}d_{s}^{22,1}d\langle B\rangle_{s}^{22}%
-\int_{0}^{t}2d_{s}^{12,1}d\langle B\rangle_{s}^{12}%
\end{align*}
and, similarly,
\[
\tilde{B}_{t}^{\Lambda,2}=B_{t}^{2}-\int_{0}^{t}d_{s}^{22,2}d\langle
B\rangle_{s}^{22}-\int_{0}^{t}b_{s}^{2}ds-\int_{0}^{t}d_{s}^{11,2}d\langle
B\rangle_{s}^{11}-\int_{0}^{t}2d_{s}^{12,2}d\langle B\rangle_{s}^{12}.
\]
Then, similarly to Proposition \ref{1dimProp}, we can derive the
representation for the solution of linear $G$-BSDE (\ref{2dimLG}).

\begin{proposition}
Suppose $(Y,Z,K)\in\mathfrak{S}_{G}^{2}(0,T)$ is the solution of the linear
$G$-BSDE (\ref{2dimLG}) satisfying the following conditions:

\begin{enumerate}
\item[(i)] $Z$ is a $2$-dimensional $G$-BMO martingale generator,

\item[(ii)] $K_{T}\in\mathbb{L}_{G}^{p}(\Omega_{T})$ for some $p>\frac{q}%
{q-1}.$
\end{enumerate}

\noindent\noindent Then we have%
\[
Y_{t}=\left(  X_{t}\right)  ^{-1}\mathbb{\hat{E}}_{t}^{\tilde{G},\Lambda
}\left[  X_{T}\xi+\int_{t}^{T}m_{s}X_{s}ds+\sum_{i,j=1}^{2}\int_{t}^{T}%
n_{s}^{ij}X_{s}d\left\langle B\right\rangle _{s}^{ij}\right]
\]
in the $8$-dimensional $\tilde{G}$-expectation space $(\tilde{\Omega}%
_{T},\mathbb{L}_{\tilde{G}}^{\Lambda,1}(\tilde{\Omega}_{T}),\mathbb{\hat{E}%
}^{\tilde{G},\Lambda})$, where $\tilde{G}$, $X$ and $\Lambda$ are given by
(\ref{1301}), (\ref{X2}) and (\ref{2dimlamda}), respectively.
\end{proposition}

\subsection{General case}

Now we deal with the general linear $G$-BSDE (\ref{LE}) driven by
$d$-dimensional $G$-Brownian motion with $G$ given by (\ref{G}). Let $X$ be
the solution to linear $G$-SDE%
\[
dX_{t}=a_{t}X_{t}dt+\sum_{i,j=1}^{d}c_{t}^{ij}X_{t}d\left\langle
B\right\rangle _{t}^{ij},~\text{ }X_{0}=1,
\]
and it has the representation
\begin{equation}
X_{t}=\exp\left(  \int_{0}^{t}a_{s}ds+\sum_{i,j=1}^{d}\int_{0}^{t}c_{s}%
^{ij}d\left\langle B\right\rangle _{s}^{ij}\right)  . \label{LGSDE}%
\end{equation}
Applying $G$-It\^{o}'s formula to $XY$ implies that%
\begin{align}
&  X_{t}Y_{t}+\int_{t}^{T}X_{s}dK_{s}+\int_{t}^{T}X_{s}Z_{s}\left(
dB_{s}-b_{s}ds-\sum_{i,j=1}^{d}d_{s}^{ij}d\left\langle B\right\rangle
_{s}^{ij}\right) \label{xy}\\
=  &  X_{T}\xi+\int_{t}^{T}m_{s}X_{s}ds+\sum_{i,j=1}^{d}\int_{t}^{T}n_{s}%
^{ij}X_{s}d\left\langle B\right\rangle _{s}^{ij}.\nonumber
\end{align}
Similarly to previous arguments for cases $d=1$ and $d=2$, our aim is to
construct an extended $\tilde{G}$-expectation space such that $B_{t}-\int
_{0}^{t}b_{s}ds-\sum_{i,j=1}^{d}\int_{0}^{t}d_{s}^{ij}d\langle B\rangle
_{s}^{ij}$ is a $\tilde{G}$-Brownian motion under Girsanov's transformation.
In fact, the dimension of the extended space should be%
\begin{equation}
\tilde{d}:=d+d+d\left(  d-1\right)  +\frac{d\left(  d-1\right)  }{2}d.
\label{ddim}%
\end{equation}

For a clearer understanding of the construction below, we give the following notations:

\begin{notation}
\begin{enumerate}
\label{Nn}

\item[(i)] For $h=1,\ldots,\frac{d(d-1)}{2}$, there is a unique pair $(i,j)$
with $i=1,\ldots,d$ and $j=i+1,\ldots,d$ such that%
\[
h=\sum_{m=0}^{i-1}\left(  d-m\right)  -d+j-i,
\]
and we denote this one-to-one correspondence by $h\sim(i,j);$

\item[(ii)] For $l=2d+1,\ldots,2d+d(d-1)$, there is a unique pair $(i,k)$ with
$i=1,\ldots,d$ and $k=1,\ldots,i-1,i+1,\ldots,d$ such that
\[
l=2d+(i-1)(d-1)+k-\mathbf{1}_{k>i},
\]
and we denote this one-to-one correspondence by $l\hat{\sim}(i,k);$

\item[(iii)] For $l=2d+d\left(  d-1\right)  +1,\ldots,\tilde{d}$, there is a
unique triplet $(i,j,k)$ with $i=1,\ldots,d-1$, $j=i+1,\ldots,d$ and
$k=1,\ldots,d$ such that%
\[
l=2d+d\left(  d-1\right)  +\left(  \sum_{m=0}^{i-1}\left(  d-m\right)
-d\right)  d+\left(  j-i-1\right)  d+k,
\]
and we denote this one-to-one correspondence by $l\bar{\sim}(i,j,k).$
\end{enumerate}
\end{notation}

\begin{example}
For $d=3$, we give the one-to-one correspondence for $h\sim(i,j)$, $l\hat
{\sim}(i,k)$ and $l\bar{\sim}(i,j,k)$: \begin{spacing}{0.6}%
\begin{align*}
&
\begin{tabular}
[c]{cccc}\toprule
$h$ & $1$ & $2$ & $3$\\\hline\addlinespace
$(i,j)$ & $(1,2)$ & $(1,3)$ & $(2,3)$\\\bottomrule
\end{tabular}
\\
&
\begin{tabular}
[c]{ccccccc}\toprule
$l$ & $7$ & $8$ & $9$ & $10$ & $11$ & $12$\\\hline\addlinespace
$(i,k)$ & $(1,2)$ & $(1,3)$ & $(2,1)$ & $(2,3)$ & $(3,1)$ & $(3,2)$\\\bottomrule
\end{tabular}
\\
&
\begin{tabular}
[c]{cccccccccc}\toprule
$l$ & $13$ & $14$ & $15$ & $16$ & $17$ & $18$ & $19$ & $20$ & $21$\\\hline\addlinespace
$(i,j,k)$ & $(1,2,1)$ & $(1,2,2)$ & $(1,2,3)$ & $(1,3,1)$ & $(1,3,2)$ &
$(1,3,3)$ & $(2,3,1)$ & $(2,3,2)$ & $(2,3,3)$\\\bottomrule
\end{tabular}
\end{align*}
\end{spacing}

\end{example}

We choose $\tilde{\Theta}\subset\mathbb{R}^{\tilde{d}\times\tilde{d}}$ as the
set of all $\tilde{d}\times\tilde{d}$ matrices of the form%
\[
\tilde{\theta}=\left(
\begin{array}
[c]{cccc}%
\theta & O_{d} & O_{d\times d\left(  d-1\right)  } & O_{d\times\frac{d\left(
d-1\right)  }{2}d}\\
\dot{\theta} & I_{d} & O_{d\times d\left(  d-1\right)  } & O_{d\times
\frac{d\left(  d-1\right)  }{2}d}\\
\hat{\theta} & O_{d\left(  d-1\right)  \times d} & I_{d\left(  d-1\right)  } &
O_{d\left(  d-1\right)  \times\frac{d\left(  d-1\right)  }{2}d}\\
\bar{\theta} & O_{\frac{d\left(  d-1\right)  }{2}d\times d} & O_{\frac
{d\left(  d-1\right)  }{2}d\times d\left(  d-1\right)  } & I_{\frac{d\left(
d-1\right)  }{2}d}%
\end{array}
\right)~  \text{ for }\theta\in\Theta\subset\mathbb{R}^{d\times d},
\]
where $\Theta$ is the set that represents $G$ in (\ref{G}), and
\[
\dot{\theta}:=\theta^{-1,\top}\in\mathbb{R}^{d\times d},~\text{ \ }\hat{\theta
}:=\left(
\begin{array}
[c]{c}%
\gamma^{11}I_{d}^{\left(  -1\right)  }\\
\vdots\\
\gamma^{dd}I_{d}^{\left(  -d\right)  }%
\end{array}
\right)  \theta^{-1,\top}\in\mathbb{R}^{d\left(  d-1\right)  \times d},~\text{
\ }\bar{\theta}:=\left(
\begin{array}
[c]{c}%
\Gamma^{1}\\
\vdots\\
\Gamma^{\frac{d\left(  d-1\right)  }{2}}%
\end{array}
\right)  \theta^{-1,\top}\in\mathbb{R}^{\frac{d\left(  d-1\right)  }{2}d\times
d}%
\]
with $(\gamma^{ij})_{i,j=1}^{d}:=$ $\theta\theta^{\top}$, $\Gamma^{h}%
:=\gamma^{ij}$ for $h\sim(i,j)$ and $I_{d}^{\left(  -i\right)  }$ being the
$\left(  d-1\right)  \times d$ submatrix of $I_{d}$ deleting the $i$th row.
For this $\tilde{\Theta}$, we define%
\begin{equation}
\tilde{G}\left(  A\right)  :=\frac{1}{2}\sup_{\tilde{\theta}\in\tilde{\Theta}%
}\operatorname*{tr}[\tilde{\theta}\tilde{\theta}^{\top}A]~\text{ for }%
A\in\mathbb{S}(\tilde{d}). \label{Gt}%
\end{equation}
Then we can construct a $\tilde{d}$-dimensional $\tilde{G}$-expectation space
$(\tilde{\Omega}_{T}:=C_{0}([0,T];\mathbb{R}^{\tilde{d}}),\mathbb{L}%
_{\tilde{G}}^{1}(\tilde{\Omega}_{T}),\mathbb{\hat{E}}^{\tilde{G}})$ with this
$\tilde{G}$ and denote by $\tilde{B}:=(\tilde{B}^{l})_{l=1}^{\tilde{d}}$ the
$\tilde{d}$-dimensional $\tilde{G}$-Brownian motion on this space. We note
that $\mathbb{\hat{E}}^{\tilde{G}}[\xi]=\mathbb{\hat{E}}[\xi]$ for each
$\xi\in\mathbb{L}_{G}^{1}\left(  \Omega\right)  $.

In order to clarify the correspondence between the quadratic variation of
$\tilde{B}$ and the set $\tilde{\Theta}\tilde{\Theta}^{\top}:=\{\tilde{\theta
}\tilde{\theta}^{\top},\tilde{\theta}\in\tilde{\Theta}\}$ as well as the one
between $\tilde{B}$ and the $\tilde{G}$-BMO martingale generator to be chosen,
we give the following notation with the help of Notation \ref{Nn}:
\begin{equation}
\tilde{B}^{l}=:\left\{
\begin{array}
[c]{ll}%
B^{l}, & \text{if }l=1,\ldots,d,\\
\dot{B}^{i}~\text{ with }i=l-d, & \text{if }l=d+1,\ldots,d+d,\\
\hat{B}^{i,k}~\text{ with }(i,k)\hat{\sim}l, & \text{if }l=2d+1,\ldots
,2d+d(d-1),\\
\bar{B}^{i,j,k}~\text{ with }(i,j,k)\bar{\sim}l,\text{ \ } & \text{if
}l=2d+d\left(  d-1\right)  +1,\ldots,\tilde{d}.
\end{array}
\right.  \label{NB}%
\end{equation}

\begin{remark}
\label{DQV}For any $\tilde{\theta}\in\tilde{\Theta}$ and its corresponding
$\theta\in\Theta$, we can check that
\begin{equation}
\tilde{\theta}\tilde{\theta}^{\top}=\left(
\begin{array}
[c]{cccc}%
\theta\theta^{\top} & \theta\dot{\theta}^{\top} & \theta\hat{\theta}^{\top} &
\theta\bar{\theta}^{\top}\\
\dot{\theta}\theta^{\top} & \dot{\theta}\dot{\theta}^{\top}+I_{d} &
\dot{\theta}\hat{\theta}^{\top} & \dot{\theta}\bar{\theta}^{\top}\\
\hat{\theta}\theta^{\top} & \hat{\theta}\dot{\theta}^{\top} & \hat{\theta}%
\hat{\theta}^{\top}+I_{d\left(  d-1\right)  } & \hat{\theta}\bar{\theta}%
^{\top}\\
\bar{\theta}\theta^{\top} & \bar{\theta}\dot{\theta}^{\top} & \bar{\theta}%
\hat{\theta}^{\top} & \bar{\theta}\bar{\theta}^{\top}+I_{\frac{d\left(
d-1\right)  }{2}d}%
\end{array}
\right)  \in\mathbb{R}^{\tilde{d}\times\tilde{d}}, \label{em}%
\end{equation}
where, using same notations as above, we have $\theta\theta^{\top}%
=(\gamma^{ij})_{i,j=1}^{d}$, $\dot{\theta}\theta^{\top}=I_{d}$ and
\[
\hat{\theta}\theta^{\top}=\left(
\begin{array}
[c]{c}%
\gamma^{11}I_{d}^{\left(  -1\right)  }\\
\vdots\\
\gamma^{dd}I_{d}^{\left(  -d\right)  }%
\end{array}
\right)  \in\mathbb{R}^{d\left(  d-1\right)  \times d},\text{ \ }\bar{\theta
}\theta^{\top}=\left(
\begin{array}
[c]{c}%
\Gamma^{1}I_{d}\\
\vdots\\
\Gamma^{\frac{d\left(  d-1\right)  }{2}}I_{d}%
\end{array}
\right)  \in\mathbb{R}^{\frac{d\left(  d-1\right)  }{2}d\times d}.
\]
Similarly to Remark \ref{2dim}, we deduce that, for any $m=1,\ldots,d$,%
\begin{align*}
\langle B^{m},\dot{B}^{i}\rangle_{t}  &  =\left\{
\begin{array}
[c]{l}%
t,\text{ if }i=m,\\
0,\text{ otherwise,}%
\end{array}
\right.  \text{ for }i=1,\ldots,d,\\
\langle B^{m},\hat{B}^{i,k}\rangle_{t}  &  =\left\{
\begin{array}
[c]{l}%
\langle B\rangle_{t}^{ii},\text{ if }k=m,\\
0,\text{ otherwise,}%
\end{array}
\right.  \text{ for }i=1,\ldots,d\text{ and }k=1,\ldots,i-1,i+1,\ldots,d,\\
\langle B^{m},\bar{B}^{i,j,k}\rangle_{t}  &  =\left\{
\begin{array}
[c]{l}%
\langle B\rangle_{t}^{ij},\text{ if }k=m,\\
0,\text{ otherwise,}%
\end{array}
\right.  \text{ for }i=1,\ldots,d-1\text{, }j=i+1,\ldots,d\text{ and
}k=1,\ldots,d.
\end{align*}

\end{remark}

Corresponding to $\tilde{B}$ in (\ref{NB}), we choose a $\tilde{d}%
$-dimensional process $\Lambda=(\Lambda^{l})_{l=1}^{\tilde{d}}$ by%
\begin{equation}
\Lambda^{l}:=\left\{
\begin{array}
[c]{ll}%
\lambda^{l}:=d^{ll,l}, & \text{if }l=1,\ldots,d,\\
\dot{\lambda}^{i}:=b^{i}~\text{ with }i=l-d, & \text{if }l=d+1,\ldots,d+d,\\
\hat{\lambda}^{i,k}:=d^{ii,k}~\text{ with }(i,k)\hat{\sim}l, & \text{if
}l=2d+1,\ldots,2d+d(d-1),\\
\bar{\lambda}^{i,j,k}:=2d^{ij,k}-d^{ii,i}\mathbf{1}_{k=j}-d^{jj,j}%
\mathbf{1}_{k=i}~\text{ with }(i,j,k)\bar{\sim}l,\text{ \ } & \text{if
}l=2d+d\left(  d-1\right)  +1,\ldots,\tilde{d},
\end{array}
\right.  \label{lamda}%
\end{equation}
which is a $\tilde{G}$-BMO martingale generator. Then, according to
Theorem \ref{Girsanov}, we know that $\tilde{B}^{\Lambda,m}=(\tilde
{B}^{\Lambda,m})_{m=1}^{\tilde{d}}$ with
\[
\tilde{B}_{t}^{\Lambda,m}:=\tilde{B}_{t}^{m}-\sum_{l=1}^{\tilde{d}}\int
_{0}^{t}\Lambda_{s}^{l}d\langle\tilde{B}\rangle_{s}^{ml},~\text{ }t\in
\lbrack0,T],\text{ }m=1,\ldots,\tilde{d},
\]
is a $\tilde{G}$-Brownian motion on the space $(\tilde{\Omega}_{T}%
,\mathbb{L}_{\tilde{G}}^{\Lambda,1}(\tilde{\Omega}_{T}),\mathbb{\hat{E}%
}^{\tilde{G},\Lambda})$, where%
\[
\mathbb{\hat{E}}^{\tilde{G},\Lambda}\left[  X\right]  =\mathbb{\hat{E}%
}^{\tilde{G}}\left[  \mathcal{E}_{T}\left(
%TCIMACRO{\tint }%
%BeginExpansion
{\textstyle\int}
%EndExpansion
\Lambda d\tilde{B}\right)  X\right]  ~\text{ and }~\mathbb{\hat{E}}_{t}%
^{\tilde{G},\Lambda}\left[  X\right]  =\mathbb{\hat{E}}_{t}^{\tilde{G}}\left[
\mathcal{E}_{T}^{t}\left(
%TCIMACRO{\tint }%
%BeginExpansion
{\textstyle\int}
%EndExpansion
\Lambda d\tilde{B}\right)  X\right]  ~\text{ for }X\in\mathbb{L}_{\tilde{G}%
}^{\Lambda,1}(\tilde{\Omega}_{T}).
\]
In fact, by (\ref{NB}), (\ref{lamda}) and Remark \ref{DQV}, one can easily
check that, for $m=1,\ldots,d,$%
\begin{align}
\tilde{B}_{t}^{\Lambda,m}=  &  B_{t}^{m}-\sum_{l=1}^{d}\int_{0}^{t}%
d_{s}^{ll,l}d\langle B\rangle_{s}^{ml}-\sum_{i=1}^{d}\int_{0}^{t}b_{s}%
^{i}d\langle B^{m},\dot{B}^{i}\rangle_{s}-\sum_{i=1}^{d}\sum_{k=1}^{d}\int
_{0}^{t}d_{s}^{ii,k}\mathbf{1}_{k\neq i}d\langle B^{m},\hat{B}^{i,k}%
\rangle_{s}\nonumber\\
&  -\sum_{\text{ }i=1}^{d-1}\sum_{\text{ }j=i+1}^{d}\sum_{k=1}^{d}\int_{0}%
^{t}\left(  2d_{s}^{ij,k}-d_{s}^{ii,i}\mathbf{1}_{k=j}-d_{s}^{jj,j}%
\mathbf{1}_{k=i}\right)  d\langle B^{m},\bar{B}^{i,j,k}\rangle_{s}\nonumber\\
=  &  B_{t}^{m}-\sum_{l=1}^{d}\int_{0}^{t}d_{s}^{ll,l}d\langle B\rangle
_{s}^{ml}-\int_{0}^{t}b_{s}^{m}ds-\sum_{i=1}^{d}\int_{0}^{t}d_{s}%
^{ii,m}\mathbf{1}_{i\neq m}d\langle B\rangle_{s}^{ii}\nonumber\\
&  -\sum_{\text{ }i=1}^{d-1}\sum_{\text{ }j=i+1}^{d}\int_{0}^{t}2d_{s}%
^{ij,m}d\langle B\rangle_{s}^{ij}+\sum_{\text{ }i=1}^{m-1}\int_{0}^{t}%
d_{s}^{ii,i}d\langle B\rangle_{s}^{im}+\sum_{\text{ }j=m+1}^{d}\int_{0}%
^{t}d_{s}^{jj,j}d\langle B\rangle_{s}^{mj}\nonumber\\
=  &  B_{t}^{m}-\int_{0}^{t}b_{s}^{m}ds-\sum_{i,j=1}^{d}\int_{0}^{t}%
d_{s}^{ij,m}d\langle B\rangle_{s}^{ij}\text{.} \label{AGBM}%
\end{align}
Thus, according to the discussion at the beginning of this subsection, we
apply conditional $\tilde{G}$-expectation $\mathbb{\hat{E}}_{t}^{\tilde
{G},\Lambda}$ on both sides of (\ref{xy}) and then obtain the main result of
this section via the same argument as in Proposition \ref{1dimProp}.

\begin{theorem}
\label{LGBSDEThm}Suppose $(Y,Z,K)\in\mathfrak{S}_{G}^{2}(0,T)$ is the solution
to the linear $G$-BSDE (\ref{LE}) satisfying the following conditions:

\begin{enumerate}
\item[(i)] $Z$ is a $d$-dimensional $G$-BMO martingale generator;

\item[(ii)] $K_{T}\in\mathbb{L}_{G}^{p}(\Omega_{T})$ for some $p>\frac{q}%
{q-1}$.
\end{enumerate}

\noindent Then we have%
\[
Y_{t}=\left(  X_{t}\right)  ^{-1}\mathbb{\hat{E}}_{t}^{\tilde{G},\Lambda
}\left[  X_{T}\xi+\int_{t}^{T}m_{s}X_{s}ds+\sum_{i,j=1}^{d}\int_{t}^{T}%
n_{s}^{ij}X_{s}d\left\langle B\right\rangle _{s}^{ij}\right]
\]
in the $\tilde{d}$-dimensional $\tilde{G}$-expectation space $(\tilde{\Omega
}_{T},\mathbb{L}_{\tilde{G}}^{\Lambda,1}(\tilde{\Omega}_{T}),\mathbb{\hat{E}%
}^{\tilde{G},\Lambda})$, where $\tilde{G}$, $\Lambda$ and $X$ are defined by
(\ref{Gt}), (\ref{lamda}) and (\ref{LGSDE}), respectively.
\end{theorem}

\begin{remark}
\label{Re28}From previous discussions, we see that, due to the nonlinear
structure, the explicit representation for the solutions to linear $G$-BSDEs
in the high dimensional $G$-expectation space is more complex than in the
$1$-dimensional case. In addition, one can check that above result also holds
when $\xi\in\mathbb{L}_{G}^{p}(\Omega_{T})$ for $p>\frac{q}{q-1}$ according to
Remark \ref{ReBMO}.
\end{remark}

\section{Infinite horizon quadratic $G$-BSDEs\label{infQGBSDE}}

In this section, the main purpose is to show the existence and uniqueness
theorem for the quadratic $G$-BSDE\ on the infinite horizon: for any $0\leq
t\leq T<\infty$,%
\begin{equation}
Y_{t}=Y_{T}+\int_{t}^{T}f\left(  s,Y_{s},Z_{s}\right)  ds+\sum_{i,j=1}^{d}%
\int_{t}^{T}g^{ij}\left(  s,Y_{s},Z_{s}\right)  d\left\langle B\right\rangle
_{s}^{ij}-\int_{t}^{T}Z_{s}dB_{s}-\left(  K_{T}-K_{t}\right)  . \label{0}%
\end{equation}
On the generators $f,g^{ij}=g^{ji}:[0,\infty)\times\Omega\times\mathbb{R}%
\times\mathbb{R}^{d}\rightarrow\mathbb{R}$, we assume the following condition
in addition to Assumption \ref{Hf}:

\begin{condition}
\label{HI}There exists a constant $\mu>0$ such that, for each $(t,\omega
)\in\lbrack0,\infty)\times\Omega$ and $z\in\mathbb{R}^{d}$,%
\[
\left(  y-y^{\prime}\right)  \left(  f\left(  t,y,z\right)  -f\left(
t,y^{\prime},z\right)  \right)  +2G\left(  \left(  y-y^{\prime}\right)
\left(  g^{ij}\left(  t,y,z\right)  -g^{ij}\left(  t,y^{\prime},z\right)
\right)  _{i,j=1}^{d}\right)  \leq-\mu\left\vert y-y^{\prime}\right\vert
^{2}.
\]

\end{condition}

Without loss of generality, we can only consider the case when $L_{y}\geq
\frac{\mu}{1-2G\left(  -J_{d}\right)  }$.

\begin{definition}
We call the triplet $(Y,Z,K)$ a solution to infinite horizon $G$-BSDE
(\ref{0}) if

\begin{enumerate}
\item[(i)] $(Y,Z,K)\in\mathfrak{S}_{G}^{2}(0,\infty)$, where $\mathfrak{S}%
_{G}^{2}(0,\infty):=\cap_{T\geq0}\mathfrak{S}_{G}^{2}(0,T)$;

\item[(ii)] it satisfies equation (\ref{0}) for all $0\leq t\leq T<\infty$.
\end{enumerate}
\end{definition}

\subsection{Some useful lemmas\label{lemmas}}

Here we give some lemmas for $G$-BSDEs on finite horizon $[0,T]$ for each
given $T>0$, which will be important in the proof of the main theorem. As a
byproduct, we also show the comparison theorem for quadratic $G$-BSDE on
finite horizon.

First, based on Section \ref{Explicit solution},\ we give the following result
which will be frequently used in the estimates of the difference of two
$G$-BSDEs. Consider the following type of equation
\begin{equation}
Y_{t}=\xi+\int_{t}^{T}f_{s}ds+\sum_{i,j=1}^{d}\int_{t}^{T}g_{s}^{ij}%
d\left\langle B\right\rangle _{s}^{ij}-\int_{t}^{T}Z_{s}dB_{s}-\left(
K_{T}-K_{t}\right)  +\left(  \bar{K}_{T}-\bar{K}_{t}\right)  , \label{dle}%
\end{equation}
where $\xi,f$ and $g^{ij}$ are given by (\ref{gen}) and (\ref{coe}) under
assumptions (i)-(iii) at the beginning of Section \ref{Explicit solution}, and
$\bar{K}$ is a nonincreasing $G$-martingale with $\bar{K}_{0}=0$ and $\bar
{K}_{T}\in\mathbb{L}_{G}^{p}(\Omega_{T})$ for some $p>\frac{q}{q-1}\vee2$.

\begin{lemma}
\label{EY}Suppose $(Y,Z,K)\in$ $\mathfrak{S}_{G}^{2}(0,T)$ is the solution to
equation (\ref{dle}) such that $Z$ is a $G$-BMO martingale generator and
$K_{T}\in${ $\mathbb{L}_{G}^{p}(\Omega_{T})$} for some $p>\frac{q}{q-1}$.
Then,
\[
Y_{t}\leq\left(  X_{t}\right)  ^{-1}\mathbb{\hat{E}}_{t}^{\tilde{G},\Lambda
}\left[  X_{T}\xi+\int_{t}^{T}m_{s}X_{s}ds+\sum_{i,j=1}^{d}\int_{t}^{T}%
n_{s}^{ij}X_{s}d\left\langle B\right\rangle _{s}^{ij}\right]
\]
in the extended $\tilde{d}$-dimensional $\tilde{G}$-expectation space
$(\tilde{\Omega}_{T},\mathbb{L}_{\tilde{G}}^{\Lambda,1}(\tilde{\Omega}%
_{T}),\mathbb{\hat{E}}^{\tilde{G},\Lambda})$, where $\tilde{G}$, $\Lambda$ and
$X$ are defined by (\ref{Gt}), (\ref{lamda}) and (\ref{LGSDE}), respectively.

Moreover, if there exists a constant $\mu>0$ such that%
\begin{equation}
a+2G\left(  \left(  c^{ij}\right)  _{i,j=1}^{d}\right)  \leq-\mu, \label{acm}%
\end{equation}
and $m,n^{ij}$ are bounded by a positive process $\rho\in\mathcal{H}_{G}%
^{p}(0,T)$ for some $p>\frac{q}{q-1}$, then%
\[
\left\vert Y_{t}\right\vert \leq\left\Vert \xi\right\Vert _{\mathbb{L}%
_{G}^{\infty}}e^{-\mu\left(  T-t\right)  }+\left(  1+\bar{\sigma}_{\Sigma}%
^{2}\right)  \mathbb{\hat{E}}_{t}^{\tilde{G},\Lambda}\left[  \int_{t}^{T}%
\rho_{s}e^{-\mu\left(  s-t\right)  }ds\right]  .
\]

\end{lemma}

\begin{proof}
First, by Proposition 3.5 in \cite{HJPS2014b}, we know that $(\bar{K}%
,0,\bar{K})\in\mathfrak{S}_{G}^{2}(0,T)$ is a solution to the linear $G$-BSDE%
\begin{align*}
Y_{t}=  &  \bar{K}_{T}+\int_{t}^{T}\left(  a_{s}Y_{s}+b_{s}Z_{s}-a_{s}\bar
{K}_{s}\right)  ds+\sum_{i,j=1}^{d}\int_{t}^{T}\left(  c_{s}^{ij}Y_{s}%
+d_{s}^{ij}Z_{s}-c_{s}^{ij}\bar{K}_{s}\right)  d\left\langle B\right\rangle
_{s}^{ij}\\
&  -\int_{t}^{T}Z_{s}dB_{s}-\left(  \bar{K}_{T}-\bar{K}_{t}\right)  ,
\end{align*}
which together with Theorem \ref{LGBSDEThm} and Remark \ref{Re28} gives that%
\begin{equation}
\bar{K}_{t}=\left(  X_{t}\right)  ^{-1}\mathbb{\hat{E}}_{t}^{\tilde{G}%
,\Lambda}\left[  X_{T}\bar{K}_{T}-\int_{t}^{T}a_{s}X_{s}\bar{K}_{s}%
ds-\sum_{i,j=1}^{d}\int_{t}^{T}c_{s}^{ij}X_{s}\bar{K}_{s}d\left\langle
B\right\rangle _{s}^{ij}\right]  . \label{1502}%
\end{equation}
Next, it is easy to see that $(Y+\bar{K},Z,K)$ satisfies the linear $G$-BSDE%
\begin{align*}
Y_{t}^{\prime}=  &  \left(  \xi+\bar{K}_{T}\right)  +\int_{t}^{T}\left(
a_{s}Y_{s}^{\prime}+b_{s}Z_{s}^{\prime}+m_{s}-a_{s}\bar{K}_{s}\right)
ds+\sum_{i,j=1}^{d}\int_{t}^{T}\left(  c_{s}^{ij}Y_{s}^{\prime}+d_{s}%
^{ij}Z_{s}^{\prime}+n_{s}^{ij}-c_{s}^{ij}\bar{K}_{s}\right)  d\left\langle
B\right\rangle _{s}^{ij}\\
&  -\int_{t}^{T}Z_{s}^{\prime}dB_{s}-\left(  K_{T}^{\prime}-K_{t}^{\prime
}\right)  .
\end{align*}
Then, applying Theorem \ref{LGBSDEThm} to above equation implies that%
\begin{equation}
Y_{t}+\bar{K}_{t}=\left(  X_{t}\right)  ^{-1}\mathbb{\hat{E}}_{t}^{\tilde
{G},\Lambda}\left[  X_{T}\left(  \xi+\bar{K}_{T}\right)  +\int_{t}^{T}\left(
m_{s}-a_{s}\bar{K}_{s}\right)  X_{s}ds+\sum_{i,j=1}^{d}\int_{t}^{T}\left(
n_{s}^{ij}-c_{s}^{ij}\bar{K}_{s}\right)  X_{s}d\left\langle B\right\rangle
_{s}^{ij}\right]  , \label{1501}%
\end{equation}
which together with (\ref{1502}) yields that%
\[
Y_{t}\leq\left(  X_{t}\right)  ^{-1}\mathbb{\hat{E}}_{t}^{\tilde{G},\Lambda
}\left[  X_{T}\xi+\int_{t}^{T}m_{s}X_{s}ds+\sum_{i,j=1}^{d}\int_{t}^{T}%
n_{s}^{ij}X_{s}d\left\langle B\right\rangle _{s}^{ij}\right]  .
\]

If (\ref{acm}) holds, then we have $X_{s}/X_{t}\leq e^{-\mu(s-t)}$ according
to (\ref{LGSDE}). Recalling that $|\mathbb{\hat{E}}_{t}^{\tilde{G},\Lambda
}[\eta+\eta^{\prime}]-\mathbb{\hat{E}}_{t}^{\tilde{G},\Lambda}[\eta^{\prime
}]|\leq\mathbb{\hat{E}}_{t}^{\tilde{G},\Lambda}[|\eta|]$ for $\eta
,\eta^{\prime}\in\mathbb{L}_{\tilde{G}}^{\Lambda,1}(\tilde{\Omega}_{T})$ by
the sublinearity of the conditional expectation, we deduce from (\ref{1502})
and (\ref{1501}) that%
\begin{align*}
\left\vert Y_{t}\right\vert \leq &  \left(  X_{t}\right)  ^{-1}\mathbb{\hat
{E}}_{t}^{\tilde{G},\Lambda}\left[  \left\vert X_{T}\xi+\int_{t}^{T}m_{s}%
X_{s}ds+\sum_{i,j=1}^{d}\int_{t}^{T}n_{s}^{ij}X_{s}d\left\langle
B\right\rangle _{s}^{ij}\right\vert \right] \\
\leq &  \left\Vert \xi\right\Vert _{\mathbb{L}_{G}^{\infty}}e^{-\mu\left(
T-t\right)  }+\left(  1+\bar{\sigma}_{\Sigma}^{2}\right)  \mathbb{\hat{E}}%
_{t}^{\tilde{G},\Lambda}\left[  \int_{t}^{T}\rho_{s}e^{-\mu\left(  s-t\right)
}ds\right]  ,
\end{align*}
which completes the proof.
\end{proof}

\bigskip

Next we give the linearization method which is similar to Lemma 3.1 in
\cite{HW2018}. We mention here we cannot directly apply the linearization
methods used in \cite{BH1998}, \cite{HJPS2014b} or \cite{HW2018} due to the
different integrability assumptions in our framework.

\begin{lemma}
\label{linearization}Let Assumptions \ref{Hf} and \ref{HI} hold. For each
given $\varepsilon>0$ and $y,y^{\prime}\in\mathbb{R}$, $z,z^{\prime}%
\in\mathbb{R}^{d}$, there exist processes $a_{s}^{\varepsilon}(y,y^{\prime
},z)$, $c_{s}^{ij,\varepsilon}(y,y^{\prime},z)$, $m_{s}^{\varepsilon
}(y,y^{\prime},z,z^{\prime})$, $n_{s}^{ij,\varepsilon}(y,y^{\prime
},z,z^{\prime})\in\mathcal{H}_{G}^{2}(0,T)$ and $b_{s}^{\varepsilon}%
(y^{\prime},z,z^{\prime})$, $d_{s}^{ij,\varepsilon}(y^{\prime},z,z^{\prime
})\in\mathcal{H}_{G}^{2}(0,T;\mathbb{R}^{d})$, $i,j=1,\ldots,d$, such that the
following properties hold:

\begin{enumerate}
\item[(i)] the generators $f$ and $g^{ij}$, $i,j=1,\ldots,d$, can be
linearized as%
\[
f\left(  s,y,z\right)  -f\left(  s,y^{\prime},z^{\prime}\right)
=a_{s}^{\varepsilon}\left(  y,y^{\prime},z\right)  \left(  y-y^{\prime
}\right)  +b_{s}^{\varepsilon}\left(  y^{\prime},z,z^{\prime}\right)  \left(
z-z^{\prime}\right)  +m_{s}^{\varepsilon}\left(  y,y^{\prime},z,z^{\prime
}\right)  ,
\]%
\[
g^{ij}\left(  s,y,z\right)  -g^{ij}\left(  s,y^{\prime},z^{\prime}\right)
=c_{s}^{ij,\varepsilon}\left(  y,y^{\prime},z\right)  \left(  y-y^{\prime
}\right)  +d_{s}^{ij,\varepsilon}\left(  y^{\prime},z,z^{\prime}\right)
\left(  z-z^{\prime}\right)  +n_{s}^{ij,\varepsilon}\left(  y,y^{\prime
},z,z^{\prime}\right)  ,
\]

\item[(ii)] these processes satisfy%
\[%
\begin{array}
[c]{ll}%
\left\vert \varphi_{s}\left(  y,y^{\prime},z\right)  \right\vert \leq
L_{y}\text{, } & \varphi=a^{\varepsilon},c^{ij,\varepsilon},\\
\left\vert \psi_{s}\left(  y^{\prime},z,z^{\prime}\right)  \right\vert
\leq2dL_{z}\left(  1+\left\vert z\right\vert +\left\vert z^{\prime}\right\vert
\right)  ,\text{ } & \psi=b^{\varepsilon},d^{ij,\varepsilon},\\
\left\vert \phi_{s}(y,y^{\prime},z,z^{\prime})\right\vert \leq2L_{y}%
\varepsilon+4dL_{z}\left(  1+\left\vert z\right\vert +\left\vert z^{\prime
}\right\vert \right)  \varepsilon, & \phi=m^{\varepsilon},n^{ij,\varepsilon},
\end{array}
\]
and%
\begin{equation}
a_{s}^{\varepsilon}\left(  y,y^{\prime},z\right)  +2G\left(  \left(
c_{s}^{ij,\varepsilon}\left(  y,y^{\prime},z\right)  \right)  _{i,j=1}%
^{d}\right)  \leq-\mu, \label{ac}%
\end{equation}

\item[(iii)] for $Y,Y^{\prime}\in\mathcal{S}_{G}^{2}(0,T)$ and $G$-BMO
martingale generators $Z,Z^{\prime}\in\mathcal{H}_{G}^{2}(0,T;\mathbb{R}^{d}%
)$,
\begin{align*}
&  a_{s}^{\varepsilon}(Y_{s},Y_{s}^{\prime},Z_{s}),c_{s}^{ij,\varepsilon
}(Y_{s},Y_{s}^{\prime},Z_{s}),m_{s}^{\varepsilon}(Y_{s},Y_{s}^{\prime}%
,Z_{s},Z_{s}^{\prime}),n_{s}^{ij,\varepsilon}(Y_{s},Y_{s}^{\prime},Z_{s}%
,Z_{s}^{\prime})\in\mathcal{H}_{G}^{2}(0,T),\\
&  b_{s}^{\varepsilon}(Y_{s}^{\prime},Z_{s},Z_{s}^{\prime}),d_{s}%
^{ij,\varepsilon}(Y_{s}^{\prime},Z_{s},Z_{s}^{\prime})\in\mathcal{H}_{G}%
^{2}(0,T;\mathbb{R}^{d}),
\end{align*}
and, moreover, $b_{s}^{\varepsilon}(Y_{s}^{\prime},Z_{s},Z_{s}^{\prime}%
),d_{s}^{ij,\varepsilon}(Y_{s}^{\prime},Z_{s},Z_{s}^{\prime}),m_{s}%
^{\varepsilon}(Y_{s},Y_{s}^{\prime},Z_{s},Z_{s}^{\prime}),n_{s}%
^{ij,\varepsilon}(Y_{s},Y_{s}^{\prime},Z_{s},Z_{s}^{\prime})$ are $G$-BMO
martingale generators.
\end{enumerate}
\end{lemma}

\begin{proof}
The main idea of the proof comes from \cite[Lemma 3.1]{HW2018} and \cite[Lemma
3.6]{HLS2018}.

\textbf{(i).} For any $\varepsilon>0$, we define
\[
l^{\varepsilon}\left(  x,x^{\prime}\right)  :=\mathbf{1}_{\left\vert
x-x^{\prime}\right\vert \geq\varepsilon}+\frac{\left\vert x-x^{\prime
}\right\vert }{\varepsilon}\mathbf{1}_{\left\vert x-x^{\prime}\right\vert
<\varepsilon}~\text{,\thinspace\ }x,x^{\prime}\in\mathbb{R}\text{,}%
\]
which is Lipschitz continuous with constant $\frac{1}{\varepsilon}$ and
satisfies $l^{\varepsilon}\left(  x,x^{\prime}\right)  \leq\frac
{1}{\varepsilon}\left\vert x-x^{\prime}\right\vert $. Then we define%
\[
a_{s}^{\varepsilon}\left(  y,y^{\prime},z\right)  :=l^{\varepsilon}\left(
y,y^{\prime}\right)  \frac{f\left(  s,y,z\right)  -f\left(  s,y^{\prime
},z\right)  }{y-y^{\prime}}-\frac{\mu}{1-2G(-J_{d})}\left(  1-l^{\varepsilon
}\left(  y,y^{\prime}\right)  \right)  ,
\]
and $b_{s}^{\varepsilon}(y^{\prime},z,z^{\prime}):=(b_{s}^{\varepsilon
,k}(y^{\prime},z,z^{\prime}))_{k=1}^{d}$ by
\begin{equation}
b_{s}^{\varepsilon,k}\left(  y^{\prime},z,z^{\prime}\right)  :=l^{\varepsilon
}\left(  z_{k},z_{k}^{\prime}\right)  \frac{f\left(  s,y^{\prime},\theta
_{k-1}\right)  -f\left(  s,y^{\prime},\theta_{k}\right)  }{z_{k}-z_{k}%
^{\prime}}+L_{z}\left(  1-l^{\varepsilon}\left(  z_{k},z_{k}^{\prime}\right)
\right)  , \label{Q3001}%
\end{equation}
where $\theta_{0}=z,$ $\theta_{k}=(z_{1}^{\prime},\ldots,z_{k}^{\prime
},z_{k+1},\ldots,z_{d})$ for $1\leq k\leq d$. Moreover, let%
\begin{equation}
m_{s}^{\varepsilon}\left(  y,y^{\prime},z,z^{\prime}\right)  :=f\left(
s,y,z\right)  -f\left(  s,y^{\prime},z^{\prime}\right)  -a_{s}^{\varepsilon
}\left(  y,y^{\prime},z\right)  \left(  y-y^{\prime}\right)  -b_{s}%
^{\varepsilon}\left(  y^{\prime},z,z^{\prime}\right)  \left(  z-z^{\prime
}\right)  . \label{m}%
\end{equation}
We can similarly define $c_{s}^{ij,\varepsilon}(y,y^{\prime},z)$,
$d_{s}^{ij,\varepsilon}(y^{\prime},z,z^{\prime})$, $n_{s}^{ij,\varepsilon
}(y,y^{\prime},z,z^{\prime})$ with respect to $g^{ij}$. It is obvious that (i)
holds and these processes belong to $\mathcal{H}_{G}^{2}(0,T)$(or
$\mathcal{H}_{G}^{2}(0,T;\mathbb{R}^{d})$) since $f,g^{ij}\in\mathcal{H}%
_{G}^{2}(0,T)$.

\textbf{(ii).} Since $0<\mu\leq(1-{2G(-}J_{d}{)})L_{y}$, one can easily check
that%
\[
\left\vert a_{s}^{\varepsilon}\left(  y,y^{\prime},z\right)  \right\vert \leq
l^{\varepsilon}\left(  y,y^{\prime}\right)  L_{y}+\frac{\mu}{1-2G(-J_{d}%
)}\left(  1-l^{\varepsilon}\left(  y,y^{\prime}\right)  \right)  \leq L_{y}.
\]
Moreover, by Assumption \ref{Hf} (iii) and the fact $|\theta_{k}%
|\leq|z|+|z^{\prime}|$, we have%
\begin{align}
\left\vert b_{s}^{\varepsilon,k}\left(  y^{\prime},z,z^{\prime}\right)
\right\vert  &  \leq l^{\varepsilon}\left(  z_{k},z_{k}^{\prime}\right)
L_{z}\left(  1+\left\vert \theta_{k-1}\right\vert +\left\vert \theta
_{k}\right\vert \right)  +L_{z}\left(  1-l^{\varepsilon}\left(  z_{k}%
,z_{k}^{\prime}\right)  \right) \nonumber\\
&  \leq L_{z}\left(  1+\left\vert \theta_{k-1}\right\vert +\left\vert
\theta_{k}\right\vert \right)  \leq2L_{z}\left(  1+\left\vert z\right\vert
+\left\vert z^{\prime}\right\vert \right)  , \label{s}%
\end{align}
and, in turn,%
\[
\left\vert b_{s}^{\varepsilon}\left(  y^{\prime},z,z^{\prime}\right)
\right\vert \leq\sum_{k=1}^{d}\left\vert b_{s}^{\varepsilon,k}\left(
y^{\prime},z,z^{\prime}\right)  \right\vert \leq2dL_{z}\left(  1+\left\vert
z\right\vert +\left\vert z^{\prime}\right\vert \right)  .
\]
Thus, we deduce that%
\begin{align*}
&  \left\vert f\left(  s,y,z\right)  -f\left(  s,y^{\prime},z\right)
-a_{s}^{\varepsilon}\left(  y,y^{\prime},z\right)  \left(  y-y^{\prime
}\right)  \right\vert \\
\leq &  \left\vert f\left(  s,y,z\right)  -f\left(  s,y^{\prime},z\right)
\right\vert \mathbf{1}_{\left\vert y-y^{\prime}\right\vert <\varepsilon
}+\left\vert a_{s}^{\varepsilon}\left(  y,y^{\prime},z\right)  \right\vert
\left\vert y-y^{\prime}\right\vert \mathbf{1}_{\left\vert y-y^{\prime
}\right\vert <\varepsilon}\leq2L_{y}\varepsilon
\end{align*}
and
\begin{align}
&  \left\vert f\left(  s,y^{\prime},z\right)  -f\left(  s,y^{\prime}%
,z^{\prime}\right)  -b_{s}^{\varepsilon}\left(  y^{\prime},z,z^{\prime
}\right)  \left(  z-z^{\prime}\right)  \right\vert \nonumber\\
\leq &  \sum_{k=1}^{d}\left\vert f\left(  s,y^{\prime},\theta_{k-1}\right)
-f\left(  s,y^{\prime},\theta_{k}\right)  -b_{s}^{\varepsilon,k}\left(
y^{\prime},z,z^{\prime}\right)  \left(  z_{k}-z_{k}^{\prime}\right)
\right\vert \nonumber\\
\leq &  \sum_{k=1}^{d}\left(  \left\vert f\left(  s,y^{\prime},\theta
_{k-1}\right)  -f\left(  s,y^{\prime},\theta_{k}\right)  \right\vert
\mathbf{1}_{\left\vert z_{k}-z_{k}^{\prime}\right\vert <\varepsilon
}+\left\vert b_{s}^{\varepsilon,k}\left(  y^{\prime},z,z^{\prime}\right)
\right\vert \left\vert z_{k}-z_{k}^{\prime}\right\vert \mathbf{1}_{\left\vert
z_{k}-z_{k}^{\prime}\right\vert <\varepsilon}\right) \nonumber\\
\leq &  4dL_{z}\varepsilon\left(  1+\left\vert z\right\vert +\left\vert
z^{\prime}\right\vert \right)  . \label{Q2901}%
\end{align}
Combining above two inequalities and the definition of $m_{s}^{\varepsilon
}(y,y^{\prime},z,z^{\prime})$ in (\ref{m}) implies that
\begin{equation}
\left\vert m_{s}^{\varepsilon}\left(  y,y^{\prime},z,z^{\prime}\right)
\right\vert \leq2L_{y}\varepsilon+4dL_{z}\varepsilon\left(  1+\left\vert
z\right\vert +\left\vert z^{\prime}\right\vert \right)  . \label{Q2902}%
\end{equation}
The proofs for $c_{s}^{ij,\varepsilon}(y,y^{\prime},z)$, $d_{s}%
^{ij,\varepsilon}(y^{\prime},z,z^{\prime})$ and $n_{s}^{ij,\varepsilon
}(y,y^{\prime},z,z^{\prime})$ are the same. Moreover, from Assumption
\ref{HI}, it is easy to see that%
\begin{align*}
&  a_{s}^{\varepsilon}\left(  y,y^{\prime},z\right)  +2G\left(  \left(
c_{s}^{ij,\varepsilon}\left(  y,y^{\prime},z\right)  \right)  _{i,j=1}%
^{d}\right) \\
\leq &  \frac{l^{\varepsilon}\left(  y,y^{\prime}\right)  }{\left\vert
y-y^{\prime}\right\vert ^{2}}\left(  \left(  y-y^{\prime}\right)  \left(
f\left(  s,y,z\right)  -f\left(  s,y^{\prime},z\right)  \right)  +2G\left(
\left(  y-y^{\prime}\right)  \left(  g^{ij}\left(  s,y,z\right)
-g^{ij}\left(  s,y^{\prime},z\right)  \right)  _{i,j=1}^{d}\right)  \right) \\
&  -\left(  1-l^{\varepsilon}\left(  y,y^{\prime}\right)  \right)  \frac{\mu
}{1-2G(-J_{d})}\left(  1-2G\left(  -J_{d}\right)  \right) \\
\leq &  -\mu.
\end{align*}

\textbf{(iii). }We first show that $a_{s}^{\varepsilon}(Y_{s},Y_{s}^{\prime
},Z_{s})\in\mathcal{H}_{G}^{2}(0,T)$ for $Y,Y^{\prime}\in\mathcal{S}_{G}%
^{2}(0,T)$ and $Z\in\mathcal{H}_{G}^{2}(0,T;\mathbb{R}^{d})$. To this end, let%
\[
f_{n}\left(  s,y,z\right)  :=f\left(  s,y,\frac{\left\vert z\right\vert \wedge
n}{\left\vert z\right\vert }z\right) ~ \text{ for }n\geq1,
\]
which is Lipschitz in $y$ and $z$ with Lipschitz constants $L_{y}$ and
$L_{z}(1+2n)$, respectively. Then, we define
\[
a_{s}^{\varepsilon,n}\left(  y,y^{\prime},z\right)  :=l^{\varepsilon}\left(
y,y^{\prime}\right)  \frac{f_{n}\left(  s,y,z\right)  -f_{n}\left(
s,y^{\prime},z\right)  }{y-y^{\prime}}-\frac{\mu}{1-2G(-J_{d})}\left(
1-l^{\varepsilon}\left(  y,y^{\prime}\right)  \right)  .
\]
We claim that $a_{s}^{\varepsilon,n}(Y_{s},Y_{s}^{\prime},Z_{s})\in
\mathcal{H}_{G}^{2}(0,T)$. In fact, for $z,\tilde{z}\in\mathbb{R}^{d}$, it is
easy to check that
\begin{align*}
\left\vert a_{s}^{\varepsilon,n}\left(  y,y^{\prime},z\right)  -a_{s}%
^{\varepsilon,n}\left(  y,y^{\prime},\tilde{z}\right)  \right\vert \leq &
\frac{l^{\varepsilon}\left(  y,y^{\prime}\right)  }{\left\vert y-y^{\prime
}\right\vert }\left(  \left\vert f_{n}\left(  s,y,z\right)  -f_{n}\left(
s,y,\tilde{z}\right)  \right\vert +\left\vert f_{n}\left(  s,y^{\prime
},z\right)  -f_{n}\left(  s,y^{\prime},\tilde{z}\right)  \right\vert \right)
\\
\leq &  \frac{2L_{z}\left(  1+2n\right)  }{\varepsilon}\left\vert z-\tilde
{z}\right\vert .
\end{align*}
Moreover, for $y,\tilde{y}\in\mathbb{R}$,
\begin{align*}
&  \left\vert a_{s}^{\varepsilon,n}\left(  y,y^{\prime},z\right)
-a_{s}^{\varepsilon,n}\left(  \tilde{y},y^{\prime},z\right)  \right\vert \\
\leq &  \left\vert \frac{l^{\varepsilon}\left(  y,y^{\prime}\right)
}{y-y^{\prime}}\left(  f_{n}\left(  s,y,z\right)  -f_{n}\left(  s,\tilde
{y},z\right)  \right)  \right\vert +\left\vert \left(  f_{n}\left(
s,\tilde{y},z\right)  -f_{n}\left(  s,y^{\prime},z\right)  \right)  \left(
\frac{l^{\varepsilon}\left(  y,y^{\prime}\right)  }{y-y^{\prime}}%
-\frac{l^{\varepsilon}\left(  \tilde{y},y^{\prime}\right)  }{\tilde
{y}-y^{\prime}}\right)  \right\vert \\
&  +\frac{\mu}{1-2G(-J_{d})}\left\vert l^{\varepsilon}\left(  y,y^{\prime
}\right)  -l^{\varepsilon}\left(  \tilde{y},y^{\prime}\right)  \right\vert \\
\leq &  \frac{L_{y}}{\varepsilon}\left\vert y-\tilde{y}\right\vert
+L_{y}\left(  \left\vert \frac{l^{\varepsilon}\left(  y,y^{\prime}\right)
}{y-y^{\prime}}\right\vert \left\vert y-\tilde{y}\right\vert +\left\vert
l^{\varepsilon}\left(  y,y^{\prime}\right)  -l^{\varepsilon}\left(  \tilde
{y},y^{\prime}\right)  \right\vert \right)  +\frac{L_{y}}{\varepsilon
}\left\vert y-\tilde{y}\right\vert \\
\leq &  \frac{4L_{y}}{\varepsilon}\left\vert y-\tilde{y}\right\vert ,
\end{align*}
and, similarly, we also have
\[
\left\vert a_{s}^{\varepsilon,n}\left(  y,y^{\prime},z\right)  -a_{s}%
^{\varepsilon,n}\left(  y,\tilde{y}^{\prime},z\right)  \right\vert \leq
\frac{4L_{y}}{\varepsilon}\left\vert y^{\prime}-\tilde{y}^{\prime}\right\vert
.
\]
Above inequalities imply that $a_{s}^{\varepsilon,n}$ is uniformly Lipschitz
in $(y,y^{\prime},z)$ for each given $\varepsilon$ and $n$. Applying
\cite[Lemma 3.3]{L2020} prove that $a_{s}^{\varepsilon,n}(Y_{s},Y_{s}^{\prime
},Z_{s})\in\mathcal{H}_{G}^{2}(0,T)$.

We note that%
\[
\left\vert a_{s}^{\varepsilon}\left(  Y_{s},Y_{s}^{\prime},Z_{s}\right)
-a_{s}^{\varepsilon,n}\left(  Y_{s},Y_{s}^{\prime},Z_{s}\right)  \right\vert
\leq\frac{1}{\varepsilon}\left(  \left\vert f\left(  s,Y_{s},Z_{s}\right)
-f_{n}\left(  s,Y_{s},Z_{s}\right)  \right\vert +\left\vert f\left(
s,Y_{s}^{\prime},Z_{s}\right)  -f_{n}\left(  s,Y_{s}^{\prime},Z_{s}\right)
\right\vert \right)  ,
\]
where we have%
\begin{equation}
\left\vert f\left(  s,Y_{s},Z_{s}\right)  -f_{n}\left(  s,Y_{s},Z_{s}\right)
\right\vert \leq L_{z}\left(  1+n+\left\vert Z_{s}\right\vert \right)  \left(
\left\vert Z_{s}\right\vert -n\right)  \mathbf{1}_{\left\vert Z_{s}\right\vert
>n}\leq2L_{z}\left\vert Z_{s}\right\vert ^{2}\mathbf{1}_{\left\vert
Z_{s}\right\vert >n}, \label{Q3003}%
\end{equation}
which also holds when $Y$ is replaced by $Y^{\prime}$. It follows that%
\[
\left\vert a_{s}^{\varepsilon}\left(  Y_{s},Y_{s}^{\prime},Z_{s}\right)
-a_{s}^{\varepsilon,n}\left(  Y_{s},Y_{s}^{\prime},Z_{s}\right)  \right\vert
\leq\frac{4L_{z}}{\varepsilon}\left\vert Z_{s}\right\vert ^{2}\mathbf{1}%
_{\left\vert Z_{s}\right\vert >n}.
\]
Therefore, from \cite[Theorem 4.7]{HWZ2016}, we deduce that
\[
\mathbb{\hat{E}}\left[  \int_{0}^{T}\left\vert a_{s}^{\varepsilon}\left(
Y_{s},Y_{s}^{\prime},Z_{s}\right)  -a_{s}^{\varepsilon,n}\left(  Y_{s}%
,Y_{s}^{\prime},Z_{s}\right)  \right\vert ds\right]  \leq\frac{4L_{z}%
}{\varepsilon}\mathbb{\hat{E}}\left[  \int_{0}^{T}\left\vert Z_{s}\right\vert
^{2}\mathbf{1}_{\left\vert Z_{s}\right\vert >n}ds\right]  \rightarrow0~\text{
as }~n\rightarrow\infty\text{,}%
\]
and, thus, we obtain $a_{s}^{\varepsilon}(Y_{s},Y_{s}^{\prime},Z_{s}%
)\in\mathcal{H}_{G}^{2}(0,T)$ by the boundedness of $a_{s}^{\varepsilon}%
(Y_{s},Y_{s}^{\prime},Z_{s})$.

Next, we verify that, for $Y^{\prime}\in\mathcal{S}_{G}^{2}(0,T)$ and $G$-BMO
martingale generators $Z,Z^{\prime}\in\mathcal{H}_{G}^{2}(0,T;\mathbb{R}^{d}%
)$, $b_{s}^{\varepsilon,k}(Y_{s}^{\prime},Z_{s},Z_{s}^{\prime})$ belongs to
$\mathcal{H}_{G}^{2}(0,T)$ and is a $G$-BMO martingale generator for each
$1\leq k\leq d$. Similarly to the above argument, we define $b_{s}%
^{\varepsilon,n,k}\left(  y^{\prime},z,z^{\prime}\right)  $ by replacing $f$
with $f_{n}$ in (\ref{Q3001}), and one can easily show that $b_{s}%
^{\varepsilon,n,k}\left(  Y_{s}^{\prime},Z_{s},Z_{s}^{\prime}\right)
\in\mathcal{H}_{G}^{2}(0,T)$. Then, using (\ref{Q3003}), we deduce that, for
each $1\leq k\leq d$,
\begin{align*}
&  \left\vert b_{s}^{\varepsilon,k}\left(  Y_{s}^{\prime},Z_{s},Z_{s}^{\prime
}\right)  -b_{s}^{\varepsilon,n,k}(Y_{s}^{\prime},Z_{s},Z_{s}^{\prime
})\right\vert \\
\leq &  \frac{1}{\varepsilon}\left(  \left\vert f\left(  s,Y_{s}^{\prime
},\Theta_{s}^{k-1}\right)  -f_{n}\left(  s,Y_{s}^{\prime},\Theta_{s}%
^{k-1}\right)  \right\vert +\left\vert f\left(  s,Y_{s}^{\prime},\Theta
_{s}^{k}\right)  -f_{n}\left(  s,Y_{s}^{\prime},\Theta_{s}^{k}\right)
\right\vert \right) \\
\leq &  \frac{2L_{z}}{\varepsilon}\left(  \left\vert \Theta_{s}^{k-1}%
\right\vert ^{2}\mathbf{1}_{\left\vert \Theta_{s}^{k-1}\right\vert
>n}+\left\vert \Theta_{s}^{k}\right\vert ^{2}\mathbf{1}_{\left\vert \Theta
_{s}^{k}\right\vert >n}\right)  ,
\end{align*}
where $\Theta_{s}^{0}=Z_{s},$ $\Theta_{s}^{k}=(Z_{s}^{\prime,1},\ldots
,Z_{s}^{\prime,k},Z_{s}^{k+1},\ldots,Z_{s}^{d})$. It follows again from
\cite[Theorem 4.7]{HWZ2016} that%
\begin{align*}
&  \mathbb{\hat{E}}\left[  \int_{0}^{T}\left\vert b_{s}^{\varepsilon,k}\left(
Y_{s}^{\prime},Z_{s},Z_{s}^{\prime}\right)  -b_{s}^{\varepsilon,n,k}%
(Y_{s}^{\prime},Z_{s},Z_{s}^{\prime})\right\vert ds\right] \\
\leq &  \frac{2L_{z}}{\varepsilon}\mathbb{\hat{E}}\left[  \int_{0}^{T}\left(
\left\vert \Theta_{s}^{k-1}\right\vert ^{2}\mathbf{1}_{\left\vert \Theta
_{s}^{k-1}\right\vert >n}+\left\vert \Theta_{s}^{k}\right\vert ^{2}%
\mathbf{1}_{\left\vert \Theta_{s}^{k}\right\vert >n}\right)  ds\right]
\rightarrow0~\text{ as }~n\rightarrow\infty\text{,}%
\end{align*}
which, together with the fact that $|b_{s}^{\varepsilon,k}(Y_{s}^{\prime
},Z_{s},Z_{s}^{\prime})|\leq2dL_{z}(1+|Z_{s}|+|Z_{s}^{\prime}|)$, implies that
$b_{s}^{\varepsilon,k}(Y_{s}^{\prime},Z_{s},Z_{s}^{\prime})\in\mathcal{H}%
_{G}^{2}(0,T)$ is a $G$-BMO martingale generator for each $1\leq k\leq d$.

Finally we obtain that $m_{s}^{\varepsilon}(Y_{s},Y_{s}^{\prime},Z_{s}%
,Z_{s}^{\prime})\in\mathcal{H}_{G}^{2}(0,T)$ is a $G$-BMO martingale by its
definition in (\ref{m}) as well as (\ref{Q2902}), and all the results for
$c_{s}^{ij,\varepsilon}(Y_{s},Y_{s}^{\prime},Z_{s})$, $d_{s}^{ij,\varepsilon
}(Y_{s}^{\prime},Z_{s},Z_{s}^{\prime})$ and $n_{s}^{ij,\varepsilon}%
(Y_{s},Y_{s}^{\prime},Z_{s},Z_{s}^{\prime})$ follow by same arguments.
\end{proof}

\begin{remark}
In Lemma \ref{linearization}, if we only suppose Assumption \ref{Hf} holds,
then the results still hold except inequality (\ref{ac}).
\end{remark}

Next we provide the comparison theorem for quadratic $G$-BSDE on finite
horizon, from which we can also clarify how to apply the linearization method
and the estimates above.

\begin{theorem}
\label{fcomparison}Let $(Y^{l},Z^{l},K^{l})\in\mathfrak{S}_{G}^{2}%
(0,T),l=1,2$, be the solution to $G$-BSDE (\ref{fQBSDE}), where $\xi^{l}%
\in\mathbb{L}_{G}^{\infty}(\Omega_{T})$ and the generators $f_{l},g_{l}%
=(g_{l}^{ij})_{i,j=1}^{d}$ satisfy Assumption \ref{Hf}. If $\xi^{1}\leq\xi
^{2}$, $f_{1}\leq f_{2}$ and $g_{1}\leq g_{2}$, then $Y^{1}\leq Y^{2}$.
\end{theorem}

\begin{proof}
Denote $\hat{Y}:=Y^{1}-Y^{2}$, $\hat{Z}:=Z^{1}-Z^{2}$ and $\hat{\xi}:=\xi
^{1}-\xi^{2}\leq0$. Then we have%
\[
\hat{Y}_{t}=\hat{\xi}+\int_{t}^{T}\hat{f}_{s}ds+\sum_{i,j=1}^{d}\int_{t}%
^{T}\hat{g}_{s}^{ij}d\left\langle B\right\rangle _{s}^{ij}-\int_{t}^{T}\hat
{Z}_{s}dB_{s}-\left(  K_{T}^{1}-K_{t}^{1}\right)  +\left(  K_{T}^{2}-K_{t}%
^{2}\right)  ,
\]
where%
\[
\hat{f}_{s}:=f_{1}\left(  s,Y_{s}^{1},Z_{s}^{1}\right)  -f_{2}\left(
s,Y_{s}^{2},Z_{s}^{2}\right)  \text{ and }\hat{g}_{s}^{ij}:=g_{1}^{ij}\left(
s,Y_{s}^{1},Z_{s}^{1}\right)  -g_{2}^{ij}\left(  s,Y_{s}^{2},Z_{s}^{2}\right)
.
\]
In the spirit of Lemma \ref{linearization}, for any $\varepsilon>0$, we have%
\begin{align}
\hat{f}_{s}  &  =a_{s}^{\varepsilon}\left(  Y_{s}^{1},Y_{s}^{2},Z_{s}%
^{1}\right)  \hat{Y}_{s}+b_{s}^{\varepsilon}\left(  Y_{s}^{2},Z_{s}^{1}%
,Z_{s}^{2}\right)  \hat{Z}_{s}+m_{s}^{\varepsilon}\left(  Y_{s}^{1},Y_{s}%
^{2},Z_{s}^{1},Z_{s}^{2}\right)  +\left(  f_{1}-f_{2}\right)  \left(
s,Y_{s}^{2},Z_{s}^{2}\right)  ,\label{fh}\\
\hat{g}_{s}^{ij}  &  =c_{s}^{ij,\varepsilon}\left(  Y_{s}^{1},Y_{s}^{2}%
,Z_{s}^{1}\right)  \hat{Y}_{s}+d_{s}^{ij,\varepsilon}\left(  Y_{s}^{2}%
,Z_{s}^{1},Z_{s}^{2}\right)  \hat{Z}_{s}+n_{s}^{ij,\varepsilon}\left(
Y_{s}^{1},Y_{s}^{2},Z_{s}^{1},Z_{s}^{2}\right)  +(g_{1}^{ij}-g_{2}%
^{ij})\left(  s,Y_{s}^{2},Z_{s}^{2}\right)  , \label{gh}%
\end{align}
where the coefficients satisfy the properties given in this lemma. Note that
$\hat{Z}$ is $G$-BMO martingale generator and $K_{T}^{l}\in\mathbb{L}_{G}%
^{p}(\Omega_{T})$ for $p\geq1$ according to Remark \ref{1101}. Then, by Lemmas
\ref{EY} and \ref{linearization}, one can easily check that,%
\begin{equation}
\hat{Y}_{t}\leq2\varepsilon\left(  1+\bar{\sigma}_{\Sigma}^{2}\right)
(\hat{X}_{t})^{-1}\mathbb{\hat{E}}_{t}^{\tilde{G},\Lambda}\left[  \int_{t}%
^{T}\left(  L_{y}+2dL_{z}+2dL_{z}\left(  \left\vert Z_{s}^{1}\right\vert
+\left\vert Z_{s}^{2}\right\vert \right)  \right)  \hat{X}_{s}ds\right]  ,
\label{2201}%
\end{equation}
in the extended $\tilde{G}$-expectation space $(\tilde{\Omega}_{T}%
,\mathbb{L}_{\tilde{G}}^{\Lambda,1}(\tilde{\Omega}_{T}),\mathbb{\hat{E}%
}^{\tilde{G},\Lambda})$, where $\tilde{G}$ is given by (\ref{Gt}), $\Lambda$
is defined by (\ref{lamda}) with $b_{s}=b_{s}^{\varepsilon}(Y_{s}^{2}%
,Z_{s}^{1},Z_{s}^{2})$ and $d_{s}^{ij}=d_{s}^{ij,\varepsilon}(Y_{s}^{2}%
,Z_{s}^{1},Z_{s}^{2})$ and
\begin{equation}
\hat{X}_{t}:=\exp\left(  \int_{0}^{t}a_{s}^{\varepsilon}\left(  Y_{s}%
^{1},Y_{s}^{2},Z_{s}^{1}\right)  ds+\sum\nolimits_{i,j=1}^{d}\int
\nolimits_{0}^{t}c_{s}^{ij,\varepsilon}\left(  Y_{s}^{1},Y_{s}^{2},Z_{s}%
^{1}\right)  d\langle B\rangle_{s}^{ij}\right)  . \label{xh}%
\end{equation}
Since $Z^{1},Z^{2}$ are $G$-BMO martingale generators and $a_{s}^{\varepsilon
}(Y_{s}^{1},Y_{s}^{2},Z_{s}^{1}),c_{s}^{ij,\varepsilon}(Y_{s}^{1},Y_{s}%
^{2},Z_{s}^{1})$ are uniformly bounded, by Remark \ref{ReBMO} (iii), sending
$\varepsilon\rightarrow0$ in (\ref{2201}) implies that $\hat{Y}_{t}\leq0$.
\end{proof}

\subsection{Existence and uniqueness result}

With the help of previous subsection, we can now investigate the
well-posedness of the quadratic $G$-BSDEs on the infinite horizon. We first
give the comparison theorem for infinite horizon quadratic $G$-BSDEs, from
which we can also derive the uniqueness result of the equation.

\begin{remark}
\label{Resolution}For each fixed $T>0$, (\ref{0}) can be seen as a $G$-BSDE on
finite horizon $[0,T]$. Thus, if the $Y$-component of the solution to
quadratic $G$-BSDE (\ref{0}) on the infinite horizon is uniformly bounded, for
each fixed $T>0$, we know that $Z$ is a $G$-BMO martingale generator on
$[0,T]$ and $K_{t}\in\mathbb{L}_{G}^{p}(\Omega_{t})$ for $p\geq1$ according to
the estimates in Remark \ref{1101} (see also Section 3 in \cite{PZ2013} for
more details of the estimates).
\end{remark}

\begin{theorem}
\label{comparison}Let $(Y^{l},Z^{l},K^{l})\in\mathfrak{S}_{G}^{2}%
(0,\infty),l=1,2$, be the solution to infinite horizon $G$-BSDE (\ref{0}),
where the generators $f_{l},g_{l}=(g_{l}^{ij})_{i,j=1}^{d}$ satisfy
Assumptions \ref{Hf} and \ref{HI}, and suppose $|Y^{l}|\leq M_{y}$ for a
constant $M_{y}>0$. If $f_{1}\leq f_{2}$ and $g_{1}\leq g_{2}$, then
$Y^{1}\leq Y^{2}$.
\end{theorem}

\begin{proof}
Let $\hat{Y}:=Y^{1}-Y^{2}$, $\hat{Z}:=Z^{1}-Z^{2}$. Similarly to the proof of
Theorem \ref{fcomparison}, we have, for any $0\leq t\leq T<\infty$,%
\[
\hat{Y}_{t}=\hat{Y}_{T}+\int_{t}^{T}\hat{f}_{s}ds+\sum_{i,j=1}^{d}\int_{t}%
^{T}\hat{g}_{s}^{ij}d\left\langle B\right\rangle _{s}^{ij}-\int_{t}^{T}\hat
{Z}_{s}dB_{s}-\left(  K_{T}^{1}-K_{t}^{1}\right)  +\left(  K_{T}^{2}-K_{t}%
^{2}\right)
\]
with $\hat{f}$ and $\hat{g}^{ij}$ given as (\ref{fh}) and (\ref{gh}) for any
$\varepsilon>0$ according to Lemma \ref{linearization}. By Lemma \ref{EY},
Remark \ref{Resolution} and the fact that $f_{1}\leq f_{2}$ and $g_{1}\leq
g_{2}$, one can easily check that%
\begin{align*}
\hat{Y}_{t}\leq &  (\hat{X}_{t})^{-1}\mathbb{\hat{E}}_{t}^{\tilde{G},\Lambda
}\left[  \hat{X}_{T}\hat{Y}_{T}+\int_{t}^{T}\hat{m}_{s}^{\varepsilon}\hat
{X}_{s}ds+\int_{t}^{T}\hat{n}_{s}^{ij,\varepsilon}\hat{X}_{s}d\left\langle
B\right\rangle _{s}^{ij}\right] \\
\leq &  2M_{y}(\hat{X}_{t})^{-1}\mathbb{\hat{E}}_{t}^{\tilde{G},\Lambda}%
[\hat{X}_{T}]+2\varepsilon\left(  L_{y}+2dL_{z}\right)  \left(  1+\bar{\sigma
}_{\Sigma}^{2}\right)  (\hat{X}_{t})^{-1}\mathbb{\hat{E}}_{t}^{\tilde
{G},\Lambda}\left[  \int_{t}^{T}\hat{X}_{s}ds\right] \\
&  +4\varepsilon dL_{z}\left(  1+\bar{\sigma}_{\Sigma}^{2}\right)  (\hat
{X}_{t})^{-1}\mathbb{\hat{E}}_{t}^{\tilde{G},\Lambda}\left[  \int_{t}%
^{T}\left(  \left\vert Z_{s}^{1}\right\vert +\left\vert Z_{s}^{2}\right\vert
\right)  \hat{X}_{s}ds\right]  ,
\end{align*}
where $\tilde{G}$, $\Lambda$ and $\hat{X}$ are the same as those in the proof
of Theorem \ref{fcomparison}. From (\ref{ac}), we know that $\hat{X}_{s}%
/\hat{X}_{t}\leq$ $e^{-\mu\left(  s-t\right)  }$ for all $s\geq t$. Thus,
\begin{align*}
\hat{Y}_{t}\leq &  2M_{y}e^{-\mu\left(  T-t\right)  }+\frac{2\varepsilon
\left(  L_{y}+2dL_{z}\right)  \left(  1+\bar{\sigma}_{\Sigma}^{2}\right)
}{\mu}\left(  1-e^{-\mu\left(  T-t\right)  }\right) \\
&  +4\varepsilon dL_{z}\left(  1+\bar{\sigma}_{\Sigma}^{2}\right)
\mathbb{\hat{E}}_{t}^{\tilde{G},\Lambda}\left[  \int_{t}^{T}\left(  \left\vert
Z_{s}^{1}\right\vert +\left\vert Z_{s}^{2}\right\vert \right)  e^{-\mu\left(
s-t\right)  }ds\right]  ,
\end{align*}
where $\mathbb{\hat{E}}_{t}^{\tilde{G},\Lambda}[%
%TCIMACRO{\tint _{t}^{T}}%
%BeginExpansion
{\textstyle\int_{t}^{T}}
%EndExpansion
(|Z_{s}^{1}|+|Z_{s}^{2}|)e^{-\mu\left(  s-t\right)  }ds]$ is finite according
to Remark \ref{ReBMO}. Therefore, sending $\varepsilon\rightarrow0$ gives
that
\[
\hat{Y}_{t}\leq2M_{y}e^{-\mu\left(  T-t\right)  }~\text{ for all }0\leq t\leq
T<\infty\text{,}%
\]
and, in turn, $Y^{1}\leq Y^{2}$.
\end{proof}

Now we are ready to give the main result of this section.

\begin{theorem}
\label{existence}Let Assumptions \ref{Hf} and \ref{HI} hold. Then the $G$-BSDE
(\ref{0}) on the infinite horizon admits a unique solution $(Y,Z,K)\in
\mathfrak{S}_{G}^{2}(0,\infty)$ such that $Y$ is uniformly bounded.
\end{theorem}

\begin{proof}
The uniqueness of the solution is a direct consequence of Theorem
\ref{comparison}.

Now we prove the existence based on the solvability of quadratic $G$-BSDEs on
finite horizon. For $n\geq1$, consider the following $G$-BSDE on $[0,n]$: for
$0\leq t\leq n$,
\begin{equation}
Y_{t}^{n}=\int_{t}^{n}f\left(  s,Y_{s}^{n},Z_{s}^{n}\right)  ds+\sum
_{i,j=1}^{d}\int_{t}^{n}g^{ij}\left(  s,Y_{s}^{n},Z_{s}^{n}\right)
d\left\langle B\right\rangle _{s}^{ij}-\int_{t}^{n}Z_{s}^{n}dB_{s}-\left(
K_{n}^{n}-K_{t}^{n}\right)  . \label{nGBSDE}%
\end{equation}
From Theorem \ref{QGBSDE} and Remark \ref{1101}, we know that it has a unique
solution $(Y^{n},Z^{n},K^{n})\in\mathfrak{S}_{G}^{2}(0,n)$ satisfying
(\ref{f1}) and (\ref{f2}). In the remainder of the proof, we begin by
verifying the uniform boundedness of $Y^{n}$ for all $n\geq1$, then
investigate the convergence of the sequence $(Y^{n},Z^{n},K^{n})$, and finally
demonstrate that the limit triplet is the desired solution to the $G$-BSDE
(\ref{0}).

First we show the boundedness of $Y^{n}$ for all $n\geq1$. Applying Lemma
\ref{linearization}, we can write (\ref{nGBSDE}) as a linear $G$-BSDE: for any
$\varepsilon>0$, we have%
\begin{align*}
f\left(  s,Y_{s}^{n},Z_{s}^{n}\right)   &  =a_{s}^{\varepsilon}\left(
Y_{s}^{n},0,Z_{s}^{n}\right)  Y_{s}^{n}+b_{s}^{\varepsilon}\left(  0,Z_{s}%
^{n},0\right)  Z_{s}^{n}+m_{s}^{\varepsilon}\left(  Y_{s}^{n},0,Z_{s}%
^{n},0\right)  +f\left(  s,0,0\right)  ,\\
g^{ij}\left(  s,Y_{s}^{n},Z_{s}^{n}\right)   &  =c_{s}^{ij,\varepsilon}\left(
Y_{s}^{n},0,Z_{s}^{n}\right)  Y_{s}^{n}+d_{s}^{ij,\varepsilon}\left(
0,Z_{s}^{n},0\right)  Z_{s}^{n}+n_{s}^{ij,\varepsilon}\left(  Y_{s}%
^{n},0,Z_{s}^{n},0\right)  +g^{ij}\left(  s,0,0\right)  ,
\end{align*}
where the coefficients satisfy the properties in Lemma \ref{linearization}. It
follows from Lemma \ref{EY} that
\begin{align*}
\left\vert Y_{t}^{n}\right\vert  &  \leq\left(  1+\bar{\sigma}_{\Sigma}%
^{2}\right)  \mathbb{\hat{E}}_{t}^{\tilde{G},\Lambda}\left[  \int_{t}%
^{n}\left(  2\varepsilon\left(  L_{y}+2dL_{z}\right)  +4dL_{z}\varepsilon
\left\vert Z_{s}^{n}\right\vert +M_{0}\right)  e^{-\mu\left(  s-t\right)
}ds\right] \\
&  \leq\frac{\left(  1+\bar{\sigma}_{\Sigma}^{2}\right)  \left(
2\varepsilon\left(  L_{y}+2dL_{z}\right)  +M_{0}\right)  }{\mu}\left(
1-e^{-\mu\left(  n-t\right)  }\right)  +4dL_{z}\varepsilon\left(
1+\bar{\sigma}_{\Sigma}^{2}\right)  \mathbb{\hat{E}}_{t}^{\tilde{G},\Lambda
}\left[  \int_{t}^{n}\left\vert Z_{s}^{n}\right\vert e^{-\mu\left(
s-t\right)  }ds\right]  ,
\end{align*}
where $\tilde{G}$ is given by (\ref{Gt}) and $\Lambda$ by (\ref{lamda}) with
$b=b_{s}^{\varepsilon}(0,Z_{s}^{n},0)$ and $d^{ij}=d_{s}^{ij,\varepsilon
}(0,Z_{s}^{n},0)$. Note that $\mathbb{\hat{E}}_{t}^{\tilde{G},\Lambda}[%
%TCIMACRO{\tint _{t}^{n}}%
%BeginExpansion
{\textstyle\int_{t}^{n}}
%EndExpansion
|Z_{s}^{n}|e^{-\mu\left(  s-t\right)  }ds]$ is finite according to Remark
\ref{ReBMO}. Thus, sending $\varepsilon\rightarrow0$ gives that
\[
\left\vert Y_{t}^{n}\right\vert \leq\frac{\left(  1+\bar{\sigma}_{\Sigma}%
^{2}\right)  M_{0}}{\mu}\left(  1-e^{-\mu\left(  n-t\right)  }\right)
\leq\frac{\left(  1+\bar{\sigma}_{\Sigma}^{2}\right)  M_{0}}{\mu}~\text{ for
all }n\geq1.
\]

Next we prove the convergence of $G$-BSDEs (\ref{nGBSDE}) as $n\rightarrow
\infty$, which allows us to obtain the solution to the $G$-BSDE (\ref{0}). For
any $n\geq1$, define%
\[
Y_{t}^{n}=0\text{, }~Z_{t}^{n}=0~\text{ and }~K_{t}^{n}=K_{n}^{n}~\text{ for
}t\geq n.
\]
Then, for $m\geq n\geq1$, $G$-BSDE (\ref{nGBSDE}) can be written as an
equation on $[0,m]$:%
\begin{align*}
Y_{t}^{n}=  &  \int_{t}^{m}\left(  f\left(  s,Y_{s}^{n},Z_{s}^{n}\right)
-f\left(  s,0,0\right)  \mathbf{1}_{s>n}\right)  ds+\sum_{i,j=1}^{d}\int
_{t}^{m}\left(  g^{ij}\left(  s,Y_{s}^{n},Z_{s}^{n}\right)  -g^{ij}\left(
s,0,0\right)  \mathbf{1}_{s>n}\right)  d\left\langle B\right\rangle _{s}%
^{ij}\\
&  -\int_{t}^{m}Z_{s}^{n}dB_{s}-\left(  K_{m}^{n}-K_{t}^{n}\right)  .
\end{align*}
On the other hand, consider the counterpart of (\ref{nGBSDE}) on $[0,m]$ with
solution $(Y^{m},Z^{m},K^{m})$ and denote $\tilde{Y}:=Y^{m}-Y^{n}$ and
$\tilde{Z}:=Z^{m}-Z^{n}$. Then we have%
\[
\tilde{Y}_{t}=\int_{t}^{m}\tilde{f}_{s}ds+\sum_{i,j=1}^{d}\int_{t}^{m}%
\tilde{g}_{s}^{ij}d\left\langle B\right\rangle _{s}^{ij}-\int_{t}^{m}\tilde
{Z}_{s}dB_{s}-\left(  K_{m}^{m}-K_{t}^{m}\right)  +\left(  K_{m}^{n}-K_{t}%
^{n}\right)  ,
\]
with%
\begin{align*}
\tilde{f}_{s}  &  :=f\left(  s,Y_{s}^{m},Z_{s}^{m}\right)  -f\left(
s,Y_{s}^{n},Z_{s}^{n}\right)  +f\left(  s,0,0\right)  \mathbf{1}_{s>n},\\
\tilde{g}_{s}^{ij}  &  :=g^{ij}\left(  s,Y_{s}^{m},Z_{s}^{m}\right)
-g^{ij}\left(  s,Y_{s}^{n},Z_{s}^{n}\right)  +g\left(  s,0,0\right)
\mathbf{1}_{s>n}.
\end{align*}
Using Lemmas \ref{linearization} and \ref{EY} again, one can easily check
that, for any $\varepsilon>0$,%
\begin{align}
|\tilde{Y}_{t}|\leq &  \left(  1+\bar{\sigma}_{\Sigma}^{2}\right)
\mathbb{\hat{E}}_{t}^{\tilde{G},\Lambda}\left[  \int_{t}^{m}\left(
2\varepsilon\left(  L_{y}+2dL_{z}\right)  +4\varepsilon dL_{z}\left(
\left\vert Z_{s}^{n}\right\vert +Z_{s}^{m}\right)  +M_{0}\mathbf{1}%
_{s>n}\right)  e^{-\mu\left(  s-t\right)  }ds\right] \nonumber\\
\leq &  \frac{2\varepsilon\left(  1+\bar{\sigma}_{\Sigma}^{2}\right)  \left(
L_{y}+2dL_{z}\right)  }{\mu}\left(  1-e^{-\mu\left(  m-t\right)  }\right)
+4\varepsilon\left(  1+\bar{\sigma}_{\Sigma}^{2}\right)  dL_{z}\mathbb{\hat
{E}}_{t}^{\tilde{G},\Lambda}\left[  \int_{t}^{m}\left(  \left\vert Z_{s}%
^{n}\right\vert +\left\vert Z_{s}^{m}\right\vert \right)  e^{-\mu\left(
s-t\right)  }ds\right] \nonumber\\
&  +\frac{\left(  1+\bar{\sigma}_{\Sigma}^{2}\right)  M_{0}e^{\mu t}}{\mu
}\left(  e^{-\mu n}-e^{-\mu m}\right)  , \label{2202}%
\end{align}
where $\tilde{G}$ is given by (\ref{Gt}) and $\Lambda$ by (\ref{lamda}) with
$b_{s}=b_{s}^{\varepsilon}(Y_{s}^{n},Z_{s}^{m},Z_{s}^{n})$ and $d_{s}%
^{ij}=d_{s}^{ij,\varepsilon}(Y_{s}^{n},Z_{s}^{m},Z_{s}^{n})$. Since
$Z^{n},Z^{m}$ are $G$-BMO martingale generators, according to Remark
\ref{ReBMO}, we have%
\[
\mathbb{\hat{E}}_{t}^{\tilde{G},\Lambda}\left[  \int_{t}^{m}\left(  \left\vert
Z_{s}^{n}\right\vert +\left\vert Z_{s}^{m}\right\vert \right)  e^{-\mu\left(
s-t\right)  }ds\right]  <\infty\text{.}%
\]
Thus, sending $\varepsilon\rightarrow0$ in (\ref{2202}) yields that%
\[
|\tilde{Y}_{t}|\leq\frac{\left(  1+\bar{\sigma}_{\Sigma}^{2}\right)
M_{0}e^{\mu t}}{\mu}\left(  e^{-\mu n}-e^{-\mu m}\right)  ,
\]
which implies that, for $p\geq1$,%
\[
\lim_{m,n\rightarrow\infty}\mathbb{\hat{E}}\left[  \sup_{t\in\left[
0,T\right]  }\left\vert Y_{t}^{m}-Y_{t}^{n}\right\vert ^{p}\right]  =0~\text{
for }T>0\text{.}%
\]
Therefore, there exists a $Y\in\mathcal{S}_{G}^{p}(0,\infty)$ such that
\begin{equation}
\left\vert Y\right\vert \leq\frac{\left(  1+\bar{\sigma}_{\Sigma}^{2}\right)
M_{0}}{\mu}~\text{ \ and }~\lim_{n\rightarrow\infty}\mathbb{\hat{E}}\left[
\sup_{t\in\left[  0,T\right]  }\left\vert Y_{t}^{n}-Y_{t}\right\vert
^{p}\right]  =0~\text{ for }T\geq0. \label{Y}%
\end{equation}
Recalling Proposition 3.7 in \cite{HLS2018}, for any $2\leq p^{\prime}<2p$ and
any $T>0$, there exists a constant $C>0$ depending on $G$ and $p$ such that
\[
\mathbb{\hat{E}}\left[  \left(  \int_{0}^{T}\left\vert Z_{t}^{m}-Z_{t}%
^{n}\right\vert ^{2}dt\right)  ^{p^{\prime}/2}\right]  \leq C\mathbb{\hat{E}%
}\left[  \sup_{t\in\left[  0,T\right]  }\left\vert Y_{t}^{m}-Y_{t}%
^{n}\right\vert ^{p}\right]  ^{p^{\prime}/2p}\rightarrow0~\text{ as
}m,n\rightarrow\infty.
\]
It follows that there exists a $Z\in\mathcal{H}_{G}^{p^{\prime}}%
(0,\infty;\mathbb{R}^{d})$ such that%
\begin{equation}
\lim_{n\rightarrow\infty}\mathbb{\hat{E}}\left[  \left(  \int_{0}%
^{T}\left\vert Z_{t}^{n}-Z_{t}\right\vert ^{2}dt\right)  ^{p^{\prime}%
/2}\right]  =0~\text{ for }T\geq0. \label{Z}%
\end{equation}

For above $Y$ and $Z$, define%
\[
K_{t}:=Y_{t}-Y_{0}+\int_{0}^{t}f\left(  s,Y_{s},Z_{s}\right)  ds+\sum
_{i,j=1}^{d}\int_{0}^{t}g^{ij}\left(  s,Y_{s},Z_{s}\right)  d\left\langle
B\right\rangle _{s}^{ij}-\int_{0}^{t}Z_{s}dB_{s}%
\]
and, in turn, we have%
\begin{align}
&  \left\vert K_{t}^{n}-K_{t}\right\vert \leq\left\vert Y_{t}^{n}%
-Y_{t}\right\vert +\left\vert Y_{0}^{n}-Y_{0}\right\vert \label{07}\\
&  +\left(  1+\bar{\sigma}_{\Sigma}^{2}\right)  \int_{0}^{t}\left(
L_{y}\left\vert Y_{s}^{n}-Y_{s}\right\vert +L_{z}\left(  1+\left\vert
Z_{s}^{n}\right\vert +\left\vert Z_{s}\right\vert \right)  \left\vert
Z_{s}^{n}-Z_{s}\right\vert \right)  ds+\left\vert \int_{0}^{t}\left(
Z_{s}^{n}-Z_{s}\right)  dB_{s}\right\vert .\nonumber
\end{align}
Note that, for $2\leq p^{\prime}\leq2p,$%
\begin{equation}
\mathbb{\hat{E}}\left[  \left(  \int_{0}^{t}\left\vert Y_{s}^{n}%
-Y_{s}\right\vert ds\right)  ^{p^{\prime}/2}\right]  \leq t^{^{p^{\prime}/2}%
}\mathbb{\hat{E}}\left[  \sup_{s\in\left[  0,t\right]  }\left\vert Y_{s}%
^{n}-Y_{s}\right\vert ^{p^{\prime}/2}\right]  , \label{0701}%
\end{equation}
and, from Burkholder-Davis-Gundy inequality,%
\begin{equation}
\mathbb{\hat{E}}\left[  \left\vert \int_{0}^{t}\left(  Z_{s}^{n}-Z_{s}\right)
dB_{s}\right\vert ^{p^{\prime}/2}\right]  \leq C_{G,p^{\prime}}\mathbb{\hat
{E}}\left[  \left(  \int_{0}^{t}\left\vert Z_{s}^{n}-Z_{s}\right\vert
^{2}ds\right)  ^{p^{\prime}/4}\right]  . \label{0702}%
\end{equation}
Moreover, by H\"{o}lder's inequality, we have%
\begin{align}
&  \mathbb{\hat{E}}\left[  \left(  \int_{0}^{t}\left(  1+\left\vert Z_{s}%
^{n}\right\vert +\left\vert Z_{s}\right\vert \right)  \left\vert Z_{s}%
^{n}-Z_{s}\right\vert ds\right)  ^{p^{\prime}/2}\right] \nonumber\\
\leq &  \mathbb{\hat{E}}\left[  \left(  \int_{0}^{t}\left(  1+\left\vert
Z_{s}^{n}-Z_{s}\right\vert +2\left\vert Z_{s}\right\vert \right)  \left\vert
Z_{s}^{n}-Z_{s}\right\vert ds\right)  ^{p^{\prime}/2}\right] \nonumber\\
\leq &  C\Bigg(\mathbb{\hat{E}}\left[  \left(  \int_{0}^{t}\left\vert
Z_{s}^{n}-Z_{s}\right\vert ds\right)  ^{p^{\prime}/2}\right]  +\mathbb{\hat
{E}}\left[  \left(  \int_{0}^{t}\left\vert Z_{s}^{n}-Z_{s}\right\vert
^{2}ds\right)  ^{p^{\prime}/2}\right]  \Bigg.\nonumber\\
&  \Bigg.+\mathbb{\hat{E}}\left[  \left(  \int_{0}^{t}\left\vert
Z_{s}\right\vert ^{2}ds\right)  ^{p^{\prime}/2}\right]  ^{1/2}\mathbb{\hat{E}%
}\left[  \left(  \int_{0}^{t}\left\vert Z_{s}^{n}-Z_{s}\right\vert
^{2}ds\right)  ^{p^{\prime}/2}\right]  ^{1/2}\Bigg). \label{0703}%
\end{align}
Thus, combining (\ref{Y})-(\ref{0703}), we deduce that%
\[
\lim_{n\rightarrow\infty}\mathbb{\hat{E}}\left[  \left\vert K_{t}^{n}%
-K_{t}\right\vert ^{p^{\prime}/2}\right]  =0~\text{ for }t\geq0\text{,}%
\]
namely, $K_{t}\in\mathbb{L}_{G}^{p^{\prime}/2}(\Omega_{t})$ for $t\geq0$.
Following the argument in Theorem 3.1 in \cite{HW2018}, we derive that $K$ is
a nonincreasing $G$-martingale. Finally, since $p$ and $p^{\prime}$ can be
arbitrary, we conclude that $(Y,Z,K)\in\mathfrak{S}_{G}^{2}(0,\infty)$
satisfies $G$-BSDE (\ref{0}).
\end{proof}

\begin{remark}
Using above argument, we can also deal with following type of Markovian
$G$-BSDEs in the infinite horizon: for any $0\leq t\leq T<\infty$,%
\begin{equation}
Y_{t}=Y_{T}+\int_{t}^{T}f\left(  X_{s},Y_{s},Z_{s}\right)  ds+\sum_{i,j=1}%
^{d}\int_{t}^{T}g^{ij}\left(  X_{s},Y_{s},Z_{s}\right)  d\left\langle
B\right\rangle _{s}^{ij}-\int_{t}^{T}Z_{s}dB_{s}-\left(  K_{T}-K_{t}\right)
,\label{1213}%
\end{equation}
where $X$ is the solution to a well-posed $G$-SDE. In fact, based on the
solvability of Markovian $G$-BSDEs on the finite horizon obtained in
\cite[Lemma 6.1]{CT2020}, the solution to (\ref{1213}) can be constructed via
a same approach in the proof of Theorem \ref{existence}. Under different
conditions, this type of infinite horizon $G$-BSDEs was also studied by
\cite{SW2024} using a truncation method, where the $Z$-component of the
solution is indeed bounded.
\end{remark}

\section*{Acknowledgements}
\addcontentsline{toc}{section}{Acknowledgements}
Yiqing Lin acknowledges the financial support from the National Natural Science Foundation of China (No. 12371473 and No. 12326603).  Falei Wang acknowledges the financial support from the National Natural Science Foundation of China (No. 12031009 and No. 12171280) and the Natural Science Foundation of Shandong Province (No. ZR2021YQ01 and No. ZR2022JQ01).

\end{document}